\documentclass[11pt,reqno]{amsart}
\usepackage{amssymb,amsmath,amsfonts,amsthm,enumerate,stmaryrd}
\usepackage[left=.75 in, right=.75 in,top=.75 in, bottom=.75 in]{geometry}

\providecommand{\abs}[1]{\left\vert#1\right\vert}
\providecommand{\norm}[1]{\left\Vert#1\right\Vert}

\providecommand{\br}[1]{\langle #1 \rangle}
\providecommand{\ns}[1]{\norm{#1}^2}

\def\nab{\nabla}
\def\dt{\partial_t}
\def\hal{\frac{1}{2}}

\def\vchi{\text{\large{$\chi$}}}
\def\ls{\lesssim}
\def\p{\partial}
\def\pa{\partial^\alpha}
\def\sg{\mathbb{D}}
\def\naba{\nab_{\mathcal{A}}}
\def\diva{\diverge_{\mathcal{A}}}

\def\hhh{{_0}H^1(\Omega)}

\def\a{\mathcal{A}}

\def\f{\mathcal{F}}
\def\g{\mathcal{G}}

\def\n{\mathcal{N}}

\def\w{\mathcal{W}}

\def\y{\mathcal{Y}}

\def\H{\mathcal{H}}

\def\k{\mathcal{K}}
\def\N{\mathbb{N}}

\def\R{\mathbb{R}}
\def\S{\mathcal{S}}
\def\T{\mathbb{T}}

\def\se{\mathcal{E}}
\def\seb{\bar{\se}}
\def\sd{\mathcal{D}}
\def\sdb{\bar{\sd}}

\def\ks{K(\sigma,\gamma)}
\def\cs{C(\sigma,\gamma)}

\def\st{\;\vert\;}

\DeclareMathOperator{\diverge}{div}

\title[Shear-flow stability]{Asymptotic stability of shear-flow solutions to incompressible viscous free boundary problems with and without surface tension}

\author{Ian Tice}
\address{
Department of Mathematical Sciences\\
Carnegie Mellon University\\
Pittsburgh, PA 15213, USA
}
\email[I. Tice]{iantice@andrew.cmu.edu}
\thanks{I. Tice was supported by a Simons Foundation Grant (\#401468) and an NSF CAREER Grant (DMS \#1653161).  This work was initiated at the Institute for Computational and Experimental Research in Mathematics (ICERM) during the Spring 2017 semester program ``Singularities and Waves In Incompressible Fluids,'' which was supported by an NSF Grant (DMS \#1439786). }

\subjclass[2010]{Primary: 35Q30, 35R35, 76E17; Secondary: 35B40, 76D45, 76E05 }

\keywords{Free boundary problems, Viscous surface waves, shear flows}

\newtheorem{lem}{Lemma}[section]

\newtheorem{prop}[lem]{Proposition}
\newtheorem{thm}[lem]{Theorem}
\newtheorem{remark}[lem]{Remark}

\theoremstyle{definition}

\numberwithin{equation}{section} 

\begin{document}

\begin{abstract}
This paper concerns the dynamics of a layer of incompressible viscous fluid lying above a rigid plane and with an upper boundary given by a free surface.  The fluid is subject to a constant external force with a horizontal component, which arises in modeling the motion of such a fluid down an inclined plane, after a coordinate change.  We consider the problem both with and without surface tension for horizontally periodic flows.  This problem gives rise to shear-flow equilibrium solutions, and the main thrust of this paper is to study the asymptotic stability of the equilibria in certain parameter regimes.  We prove that there exists a parameter regime in which sufficiently small perturbations of the equilibrium at time $t=0$ give rise to global-in-time solutions that return to equilibrium exponentially in the case with surface tension and almost exponentially in the case without surface tension.  We also establish a vanishing surface tension limit, which connects the solutions with and without surface tension.   
\end{abstract}

\maketitle


\section{Introduction }

\subsection{Free boundary Navier-Stokes equations}

Consider a layer of viscous incompressible fluid evolving above a flat plane in three dimensions.  We assume that the fluid is subjected to a uniform force field of the form $f =(\gamma, 0 , -g)  = \gamma e_1 - g e_3 \in \R^3$, where $\gamma \ge 0$ and $g > 0$ are constants.  Such a force arises, for instance, if we consider a fluid sliding down an inclined plane, subject to a constant gravitational field $G \in \R^3$, and we change coordinates to view the plane as orthogonal to $e_3$.  In this case the constants $\gamma,g$ can be obtained by resolving $G$ into components perpendicular to the plane (corresponding to $-g e_3$ here) and along the plane (corresponding to $\gamma e_1$ here).  We assume that $g>0$ in order to specify that $G$ is not purely tangential to the plane.  

In addition to the above assumption on the external force acting on the fluid, we will assume three other main features.  First, we assume that the fluid is bounded above by a free surface that evolves with the fluid.  Second, we assume that above the free interface the fluid is bordered by a trivial fluid of constant pressure (for instance a vacuum). Third, we assume that the fluid is horizontally periodic so that we can determine its dynamics by studying a single horizontal periodicity cell.

Let us now state the equations of motion for the problem.  We will model the periodicity of the fluid by introducing the horizontal cross section 
\begin{equation}
\Sigma = (L_1 \mathbb{T}) \times (L_2 \mathbb{T}),  
\end{equation}
where $L_1, L_2 >0$ are horizontal periodicity lengths and $L_i \mathbb{T} = \R / (L_i\mathbb{Z})$ is a standard flat $1-$torus of periodicity $L_i$.  We will assume that the moving upper boundary of the fluid is given by the graph of an unknown function $\eta: \Sigma \times [0,\infty) \to \R$, which means that the moving fluid domain is modeled by the three dimensional set
\begin{equation}
\Omega(t) = \{ x = (x',x_3) \in \Sigma \times \R \st -b < x_3 < \eta(x',t)\},
\end{equation}
where $b>0$ is a constant depth parameter.  Note that the lower boundary of $\Omega(t)$ is the fixed and unmoving set 
\begin{equation}
 \Sigma_b = \{ x= (x',x_3) \in \Sigma \times \R \st  x_3 = -b\},
\end{equation}
while the moving upper surface is 
\begin{equation}
 \Sigma(t) = \{ x =(x',x_3) \in \Sigma \times \R \st  x_3 = \eta(x',t)\}.
\end{equation}

For each $t\ge 0$ the fluid is described by its velocity and pressure functions $(\bar{u},\bar{p}) :\Omega(t) \to \R^3 \times \R$.  We require that $(\bar{u}, \bar{p}, \eta)$ satisfy the  incompressible
Navier-Stokes equations in $\Omega(t)$ for $t>0$: 
\begin{equation}\label{ns_euler}
\begin{cases}
\partial_t \bar{u} + \bar{u} \cdot \nabla \bar{u} + \nabla \bar{p} = \mu \Delta \bar{u} + \gamma e_1 - g e_3 & \text{in }
\Omega(t) \\ 
\diverge{\bar{u}}=0 & \text{in }\Omega(t) \\ 
\partial_t \eta = \bar{u}_3 - \bar{u}_1 \partial_{1}\eta - \bar{u}_2 \partial_{2}\eta & 
\text{on }\Sigma(t) \\ 
(\bar{p} I -  \mu \mathbb{D}(\bar{u}) ) \nu = (P_{ext}   - \sigma \mathfrak{H}(\eta))\nu & \text{on } \Sigma(t) \\ 
\bar{u} = 0 & \text{on } \Sigma_b.
\end{cases}
\end{equation}
Here $\mu >0$ is the fluid viscosity,  $(\mathbb{D} \bar{u})_{ij} = \partial_i \bar{u}_j + \partial_j \bar{u}_i$ the symmetric gradient of $\bar{u}$,  $\nu$ is the outward-pointing unit normal vector on $\Sigma(t)$,  $I$ the $3 \times 3$ identity matrix, $P_{ext} \in \R$ is the constant pressure above the fluid, $\sigma \ge 0$ is the surface tension coefficient, and 
\begin{equation}\label{H_def}
 \mathfrak{H}(\eta) = \diverge\left( \frac{\nab \eta}{\sqrt{1+\abs{\nab \eta}^2}}\right)
\end{equation}
is (minus) twice the mean curvature of $\Sigma(t)$.  The first two equations in \eqref{ns_euler} are the standard incompressible Navier-Stokes equations, the third is the kinematic transport equation for $\eta$, the fourth is the balance of stress at the interface, and the fourth is the no-slip boundary condition at the bottom.  The problem is augmented with initial data $\eta_0 : \Sigma \to (-b,\infty)$ which determines the initial domain $\Omega_0$, as well as an initial velocity field $\bar{u}_0 : \Omega_0 \to \R^3$.  Note that the assumption  $\eta_0 > -b$ on $\Sigma$ means that $\Omega_0$ is well-defined.   

Without loss of generality, we may assume that $\mu = g = 1$.  Indeed, a standard scaling argument allows us to rescale so that $\mu = g =1$, at the price of multiplying the parameters $\gamma,$ $\sigma,$ $b$, and the periodicity lengths $L_1,L_2$ by positive constants.  This means that, up to renaming  $\gamma,$ $\sigma,$ $b$, $L_1$, and $L_2$, we arrive at \eqref{ns_euler} with $\mu=g=1$.

We will also assume that the initial surface function satisfies the ``zero average'' condition 
\begin{equation}\label{z_avg}
\frac{1}{L_1 L_2} \int_\Sigma \eta_0 =0.
\end{equation}
This is not a loss of generality due to the assumption that $\eta_0 > -b$ on $\Sigma$: see the introduction of \cite{guo_tice_per} for an explanation of how to obtain this condition via a coordinate shift.  Note that for sufficiently regular solutions to \eqref{ns_euler}, the zero average condition persists in time since $\dt \eta = \bar{u} \cdot \nu \sqrt{1 + (\p_{1} \eta)^2 + (\p_{2} \eta)^2}$, and hence
\begin{equation}\label{avg_prop}
 \frac{d}{dt}  \int_{\Sigma} \eta =  \int_{\Sigma} \dt \eta  = \int_{ \Sigma(t) } \bar{u} \cdot \nu = \int_{\Omega(t)} \diverge{\bar{u}} = 0.
\end{equation}

\subsection{Steady shear solution}
 
It is a simple matter to construct a steady shear-flow solution to \eqref{ns_euler}.  We set $\eta =0$ in order to reduce to the flat slab domain $\Omega = \{x \in \Sigma \times \R \st -b < x_3 < 0\}$.  We then define the smooth function $s: \R \to \R$ via  
\begin{equation}\label{shear_s_def}
 s(x_3) = \frac{\gamma}{2 }(b^2 - x_3^2).
\end{equation}
We then define the equilibrium shear velocity to be $U: \Sigma \times \R \to \R^3$ given by  
\begin{equation}\label{shear_U_def}
 U(x) = (s(x_1),0,0) = s(x_3) e_1,
\end{equation}
and we define the equilibrium hydrostatic pressure to be the smooth function $P: \Sigma \times \R \to \R$ given by 
\begin{equation}\label{shear_P_def}
 P(x) = P_{ex} -  x_3.
\end{equation}

Trivial calculations show that 
\begin{equation}\label{shear_computations}
\nab P = -g e_3, \;  U \cdot \nab U =0, \; - \Delta U = \gamma e_1, 
\text{ and }  \sg U(x) =  s'(x_3) M \text{ for }
M = 
\begin{pmatrix}
 0 & 0 & 1 \\
 0 & 0 & 0 \\
 1 & 0 & 0
\end{pmatrix}
\end{equation}
on the entire set $\Sigma \times \R$, not just in $\Omega$.  Moreover, $U$ satisfies $U\vert_{\Sigma_b}=0$ and $\sg U = 0$ on $\Sigma \times \{x_3 =0\}$.  Consequently, the triple
\begin{equation}\label{shear_soln}
\bar{u} = U, \bar{p} = P, \eta =0  
\end{equation}
constitute a steady shear-flow solution to \eqref{ns_euler} (recalling that $\mu =g =1$).

\subsection{Eulerian perturbation }

A key feature of the steady shear-flow solution \eqref{shear_soln} is that the velocity $U$ and pressure $P$ are defined in all of $\Sigma \times \R$.  This allows us to easily rewrite the system \eqref{ns_euler} as a perturbation of the shear-flow solution in Eulerian coordinates.  Indeed, from \eqref{shear_computations} we see that the triple $(\bar{u},\bar{p}, \eta)$ solves \eqref{ns_euler} if and only if $\bar{u} = u + U$ and $\bar{p} = p +P$ and the triple $(u,p,\eta)$ solves 
\begin{equation}\label{ns_perturbed}
\begin{cases}
 \dt u + u \cdot \nab u + U \cdot \nab u + u \cdot \nab U + \nab p - \Delta u =0 & \text{in }\Omega(t) \\
 \diverge u = 0 & \text{in } \Omega(t) \\
 \dt \eta = u_3 - u_2 \p_2 \eta -(u_1 + s(\eta)) \p_1 \eta & \text{on } \Sigma(t) \\
 (pI - \sg u) \nu = [( \eta - \sigma \mathfrak{H}(\eta))I + s'(\eta) M] \nu & \text{on } \Sigma(t) \\
 u = 0 & \text{on } \Sigma_b,
\end{cases}
\end{equation}
where $s$ and $M$ are as defined in \eqref{shear_s_def} and \eqref{shear_computations}, respectively.  Note that in the perturbative formulation the shear velocity $U$ interacts with $u$ and $p$ in the Navier-Stokes equations and with $\eta$ and $u$ in the kinematic transport and dynamic stress balance equations.  We will study the problem in this formulation, with the aim being to show that the equilibrium solution $u=0$, $p=0$, $\eta=0$, which corresponds to the steady shear-flow solution in \eqref{ns_euler}, is asymptotically stable for some range of the parameters.

In order to justify why we might expect such a stability result, let us examine the natural energy-dissipation equation associated to \eqref{ns_perturbed}.  Reynolds' transport theorem allows us to compute 
\begin{equation}
 \frac{d}{dt} \int_{\Omega(t)} \hal\abs{u}^2 = \int_{\Omega(t)} \dt \frac{\abs{u}^2}{2} + (u+ U) \cdot \nab \frac{\abs{u}^2}{2} = \int_{\Omega(t)} (\dt u + u \cdot \nab u + U \cdot \nab u)\cdot u. 
\end{equation}
Plugging in the first equation in \eqref{ns_perturbed} then allows us to compute from this (see Proposition \ref{lin_evolve_geo} for details) that we have the following natural energy-dissipation identity for sufficiently regular solutions to \eqref{ns_perturbed}
\begin{multline}\label{nat_en_diss}
 \frac{d}{dt} \left( \int_{\Omega(t)} \hal \abs{u}^2 + \int_{\Sigma} \frac{1}{2} \abs{\eta}^2 + \sigma \sqrt{1 + \abs{\nab \eta}^2}   \right) + \int_{\Omega(t)} \hal \abs{\sg u}^2 
 = \int_{\Omega(t)} - (u\cdot \nab U) \cdot u  + \int_{\Sigma(t)} s'(\eta) M\nu \cdot u \\
 - \int_\Sigma (\eta - \sigma \mathfrak{H}(\eta)) s(\eta) \p_1 \eta.
\end{multline}
The identity \eqref{nat_en_diss} provides heuristics for when we can expect stability, namely when we can absorb the quadratic terms on the right onto the left.  In order to do this we will clearly need to work in the context of small perturbations and in a small $\gamma$ regime.  Outside of this regime it is possible that the terms on the right side of \eqref{nat_en_diss} act as a source of energy, leading to the unstable growth of the energetic term in parentheses.

\subsection{Previous work }

The problem \eqref{ns_euler} and its variants have attracted much attention in the mathematics community, so we will attempt only a brief survey of the literature.  The instability of viscous shear flows in fixed domains is a classical question, going back to the work of Orr \cite{orr} and Sommerfeld \cite{sommerfeld}, whose eponymous equation appears in the spectral theory of the linearized bulk equations.  The viscous instability of shear flows in rigid domains was developed formally in the physics literature by many authors, including  Heisenberg \cite{heisenberg}, Lin \cite{lin}, and Tollmien \cite{tollmien}.  Remarkably, a rigorous mathematical proof of long-wave instability of steady shear flows without free boundary appeared only recently in the work of Grenier-Guo-Nguyen \cite{gren_guo_nguyen}.

The equilibrium shear solution \eqref{shear_soln} for the free boundary problem \eqref{ns_euler} depends on the specific geometry of the domain, but one can seek stationary solution in different geometries as well.  Abergel-Bona \cite{abergel_bona} constructed solutions to the 2D steady Stokes problem over an infinite inclined plane with a non-flat bottom and surface tension.  Abergel-Bailly \cite{abergel_bailly} constructed steady Stokes solutions in 3D without surface tension.  Pileckas-Solonnikov \cite{pileckas_solonnikov} studied stationary Navier-Stokes flow in 2D domains with unbounded lower boundaries and surface tension.

When $\gamma =0$ the dynamics of the free boundary problem \eqref{ns_euler} are well-understood for small data.  Nishida-Teramoto-Yoshihara constructed small data solutions for the problem with surface tension and showed that the solutions exist globally and decay to equilibrium exponentially fast.  The corresponding problem without surface tension was handled by Guo-Tice \cite{guo_tice_per}, who constructed global solutions that decay almost exponentially.

When $\gamma \neq 0$ less is known about the free boundary problem \eqref{ns_euler}.  Sun \cite{sun} studied the 2D problem with fixed positive surface tension in the context of semigroups, proving stability under some assumptions on the spectrum of the linearized operator that were verified numerically.  Ueno \cite{ueno} considered the 2D problem with fixed positive surface tension in the thin film regime and derived uniform estimates with respect to the thinness parameter, valid locally in time. Padula \cite{padula_1,padula_2} studied the 3D problem with fixed positive surface tension and developed sufficient conditions for low-regularity asymptotic stability under the a priori assumption of the existence of global smooth solutions.  Teramoto-Tomoeda \cite{teramoto_tomoeda} studied the linearized 2D problem with fixed positive surface tension and proved that the linearized problem generates an analytic semigroup, but they did not discuss the nonlinear theory. Nishida-Teramoto-Yoshihara \cite{nis_ter_yos_bif} developed the Hopf bifurcation analysis for the 2D problem with fixed surface tension under some assumptions on the spectrum of the linearized operator in order to construct time periodic solutions.  To the best of our knowledge no nonlinear stability results are known without surface tension.

\subsection{Reformulation in a flattened coordinate system}

The moving domain $\Omega(t)$ is inconvenient for analysis, so we will reformulate the problem \eqref{ns_perturbed} in  the fixed equilibrium domain 
\begin{equation}
\Omega= \{x \in \Sigma \times \R \st  -b < x_3 < 0  \}.
\end{equation}
We will think of $\Sigma$ as the upper boundary of $\Omega$ and view $\eta$ as a function on $\Sigma \times [0,\infty)$.  We then define 
\begin{equation}
 \bar{\eta}:= \mathcal{P} \eta = \text{harmonic extension of }\eta \text{ into the lower half space},
\end{equation}
where $\mathcal{P} \eta$ is defined by \eqref{poisson_def_per}.  We then flatten the coordinate domain via the smooth mapping $\Phi: \Omega \times [0,\infty) \to \Sigma \times \R$ given by 
\begin{equation}\label{mapping_def}
 \Phi(x,t) = (x_1,x_2, x_3 +  \bar{\eta}(x,t)(1+ x_3/b )) \in \Omega(t).
\end{equation}
Note that $\Phi(\cdot,t)$ extends to $\bar{\Omega}$, and that $\Phi(\Sigma,t) = \Sigma(t)$ and $\Phi(\cdot,t)\vert_{\Sigma_b} = Id_{\Sigma_b}$, i.e. $\Phi$ maps $\Sigma$ to the free surface and keeps the lower surface fixed.   We have
\begin{equation}\label{A_def}
 \nab \Phi = 
\begin{pmatrix}
 1 & 0 & 0 \\
 0 & 1 & 0 \\
 A & B & J
\end{pmatrix}
\text{ and }
 \mathcal{A} := (\nab \Phi^{-1})^T = 
\begin{pmatrix}
 1 & 0 & -A K \\
 0 & 1 & -B K \\
 0 & 0 & K
\end{pmatrix}
\end{equation}
for 
\begin{equation}\label{ABJ_def}
\begin{split}
A &= \p_1 \bar{\eta} \tilde{b},\;\;\;  B = \p_2 \bar{\eta} \tilde{b},  \\ 
J &=  1+ \bar{\eta}/b + \p_3 \bar{\eta} \tilde{b},  \;\;\; K = J^{-1}, \\ 
\tilde{b}  &= (1+x_3/b).  
\end{split}
\end{equation} 
Here $J = \det{\nab \Phi}$ is the Jacobian of the coordinate transformation.

The matrix $\a$ in \eqref{A_def} allows us to define a collection of $\a-$dependent differential operators.  We define the differential operators $\naba$ and $\diva$ with their actions given by
\begin{equation}
 (\naba f)_i := \a_{ij} \p_j f \text{ and } \diva X := \a_{ij}\p_j X_i
\end{equation}
for appropriate $f$ and $X$.  We also extend $\diva$ to act on symmetric tensors in the usual way.  We also write
\begin{equation}
 (\sg_{\a} u)_{ij} =  \a_{ik} \p_k u_j + \a_{jk} \p_k u_i \text{ and }  S_\a(p,u) = pI - \sg_{\a} u,
\end{equation}
and we define
\begin{equation}\label{N_def}
\n := (-\p_1 \eta, - \p_2 \eta,1)
\end{equation}
for the non-unit normal to $\Sigma(t)$.  In the new coordinate system the system of PDEs \eqref{ns_perturbed} becomes the following system:
\begin{equation}\label{ns_geometric}
\begin{cases}
 \dt u - \dt \bar{\eta} \tilde{b} K \p_3 u + u \cdot \naba u + U \circ \Phi \cdot \naba u + u \cdot \naba (U \circ \Phi) + \diva S_{\a}(p,u) =0 & \text{in }\Omega \\
 \diva u =0 & \text{in } \Omega \\
 S_\a(p,u) \n = [(\eta - \sigma \mathfrak{H}(\eta)) I -\gamma  \eta M]\n & \text{on } \Sigma \\
 \dt \eta = u \cdot \n - s(\eta) \p_1 \eta & \text{on }\Sigma \\
 u = 0 & \text{on } \Sigma_b.
\end{cases}
\end{equation}

\section{Main results and discussion }

\subsection{Notation and definitions }\label{sec_notation}

In order to properly state our main results we must first introduce some notation and define various functionals that will be used throughout the paper.  We begin with some notational conventions.

\textbf{Constants:}  Throughout the paper $C>0$ will denote a generic constant that can depend on $\Omega$ and its dimensions as well as on $g$ and $\mu$ (though these are both scaled to unity, so the dependence on these is implicit), but \emph{not} on the parameters $\gamma,\sigma$.  Such constants are referred to as ``universal,'' and they are allowed to change from one inequality to another. We employ the notation $a \ls b$ to mean that $a \le C b$ for a universal constant $C>0$.  

We will also need to track constants that depend on $\sigma$ and $\gamma$.  To this end we introduce two pieces of notation.  When we write $\ks$ we mean a positive constant that depends on $\sigma$ and $\gamma$ in such a way that 
\begin{equation}
  \liminf_{(\sigma,\gamma)\to 0} \ks \in (0,\infty) \text{ and }
 \limsup_{(\sigma,\gamma)\to \infty} \frac{\ks}{\max\{\sigma,\gamma \}^j } < \infty \text{ for an integer }j >0.
\end{equation}
In other words, $\ks$ denotes a constant that remains positive as $\sigma,\gamma \to 0$ and grows at most like a polynomial in $\sigma,\gamma$.  We will also write $\cs$ to denote a positive constant depending on $\sigma$ and $\gamma$ such that 
\begin{equation}
  \liminf_{ \gamma\to 0} \cs \in (0,\infty) \text{ and }
 \limsup_{(\sigma,\gamma)\to \infty} \frac{\cs}{\max\{\sigma,\gamma \}^j } < \infty \text{ for an integer }j >0.
\end{equation}
The key difference is that constants $\cs$ are allowed to blow up as $\sigma \to 0$.

\textbf{Norms:} We write $H^k(\Omega)$ with $k\ge 0$ and $H^s(\Sigma)$ with $s \in \R$ for the usual $L^2-$based Sobolev spaces.  In particular $H^0 = L^2$.  In the interest of concision, we neglect to write $H^k(\Omega)$ or $H^k(\Sigma)$ in our norms and typically write only $\norm{\cdot}_{k}$.  The price we pay for this is some minor ambiguity in the set on which the norm is computed, but we mitigate potential confusion by always writing the space for the norm when traces are involved.

\textbf{Multi-indices:} We will write $\mathbb{N}^k$ for the usual set of multi-indices, where here we employ the convention that $0 \in \mathbb{N}$.  For $\alpha \in \mathbb{N}^k$ we define the spatial differential operator $\p^\alpha = \p_1^{\alpha_1} \cdots \p_k^{\alpha_k}$.  We will also write $\mathbb{N}^{1+k}$  to denote the set of space-time multi-indices 
\begin{equation}
 \mathbb{N}^{1+k} = \{(\alpha_0,\alpha_1,\dotsc,\alpha_k) \st \alpha_i \in \mathbb{N} \text{ for }0\le i \le k\},
\end{equation}
For a multi-index $\alpha \in \mathbb{N}^{1+k}$ we define the differential operator $\p^\alpha = \dt^{\alpha_0} \p_1^{\alpha_1} \cdots \p_k^{\alpha_k}$.  Also, for a space-time multi-index $\alpha \in \mathbb{N}^{1+k}$ we use the parabolic counting scheme $\abs{\alpha} = 2\alpha_1 + \alpha_1 + \cdots + \alpha_k$.

\textbf{Energy and dissipation functionals:}  Throughout the paper we will make frequent use of various energy and dissipation functionals, and we will track their dependence on an integer $n \ge 3$ and the surface tension coefficient $\sigma \ge 0$.  We define these now.  The basic and full energy functionals, respectively, are defined as:
\begin{equation}\label{E_bar_def}
 \seb_n^\sigma =   \sum_{\substack{\alpha \in \mathbb{N}^{1+2} \\\abs{\alpha} \le 2n }} \ns{ \p^\alpha u}_{0}  +  \ns{\p^\alpha \eta}_{0} + \sigma \ns{ \nab \p^\alpha \eta }_{0}   
\end{equation}
and
\begin{multline}\label{E_def}
 \se_n^\sigma = \seb_n^\sigma + \sum_{j=0}^n \ns{\dt^j u}_{2n-2j} + \sum_{j=0}^{n-1} \ns{\dt^j p}_{2n-2j-1} +  \ns{\eta}_{2n} + \sigma \ns{\eta }_{2n+1} \\
 + \ns{\dt \eta}_{2n-1} + \sigma \ns{\dt \eta}_{2n-1/2}  +\sum_{j=2}^n \ns{\dt^j \eta}_{2n-2j+3/2}. 
\end{multline}
The corresponding basic and full dissipation functionals are
\begin{equation}\label{D_bar_def}
 \sdb_n = \sum_{\substack{\alpha \in \mathbb{N}^{1+2} \\\abs{\alpha} \le 2n }} \ns{ \sg \p^\alpha u}_0
\end{equation}
and
\begin{multline}\label{D_def}
 \sd_n^\sigma =\sdb_n + \sum_{j=0}^n \ns{\dt^j u}_{2n-2j+1} + \sum_{j=0}^{n-1} \ns{\dt^j p}_{2n-2j} 
+   \sum_{j=0}^{n-1}\left((1+\gamma^2) \ns{\dt^j \eta}_{2n-2j-1/2} + \sigma^2 \ns{\dt^j \eta}_{2n-2j+3/2} \right)  \\
+ \ns{\dt \eta}_{2n-1} + \sigma^2 \ns{\dt \eta}_{2n+1/2} 
 + \ns{\dt^2 \eta}_{2n-2} + \sigma^2 \ns{\dt^2 \eta}_{2n-3/2} + \sum_{j=3}^{n+1} \ns{\dt^j \eta}_{2n-2j+5/2}.
\end{multline}
We will also need to make frequent reference to two functionals that are not naturally of energy or dissipation type.  We refer to these as
\begin{equation}\label{transport_def}
 \f_{n}=  \ns{ \eta}_{2n+1/2}
\end{equation}
and
\begin{equation}\label{K_def}
 \k = \ns{u}_{C^2_b(\Omega)} +  \ns{u}_{H^3(\Sigma)} + \ns{\eta}_{5/2}.
\end{equation}

\subsection{Local existence theory }

Before stating our results on the global existence and long-term behavior of solutions to \eqref{ns_geometric} we must first discuss the local existence theory.  For the sake of brevity we will not attempt to properly develop this theory in this paper.  This is justified by the fact that there are now numerous examples of how to use a priori estimates of the form we develop in this paper to design a scheme of approximate problems that can be used to construct local-in-time solutions to \eqref{ns_geometric} in a functional setting appropriate for the a priori estimates.   We refer to \cite{guo_tice_per,jang_tice_wang_gwp,tan_wang,wu} for four such examples.  Instead, here we will only state the local existence result that can be proved with these techniques.  

In the local existence statement we will need two extra ingredients.  The first is to define the spaces 
\begin{equation}
 \hhh  = \{ v \in H^1(\Omega; \R^3) \st v \vert_{\Sigma_b} =0\},
\end{equation}
and 
\begin{equation}
 \mathcal{X}_T = \{ u \in L^2([0,T];\hhh ) \st \diverge_{\a(t)} u(t) = 0 \text{ for a.e. }t \in [0,T]\},
\end{equation}
where here we view $\a(t)$ as determined by the $\eta: \Omega \times [0,T] \to \R$ coming from the solution.  The second ingredient is the idea of compatibility conditions.  Here the point is that in order to produce solutions to \eqref{ns_geometric} in a given regularity class, the initial data for $u$ and $\eta$ must be ``compatible'' in the sense that they satisfy a certain finite collection of equations.  These equations are simple to derive but quite cumbersome to write in full detail, so rather than attempt to do so here, we will again only refer to \cite{guo_tice_per,jang_tice_wang_gwp,tan_wang,wu}.  

The local existence result is then the following theorem.

\begin{thm}\label{lwp}
Let $n \ge 3$ be an integer and suppose that the initial data $(u_0,\eta_0)$ satisfy 
\begin{equation}
 \ns{u_0}_{2n} + \ns{\eta_0}_{2n+1/2} + \sigma \ns{\nab \eta_0}_{2n} < \infty
\end{equation}
as well as the natural compatibility conditions associated with $n$.  Then there exist $0 < \delta_\ast,T_\ast <1$ such that if 
\begin{equation}
 \ns{u_0}_{2n} + \ns{\eta_0}_{2n+1/2} + \sigma \ns{\nab \eta_0}_{2n} \le \delta_\ast
\end{equation}
and $0 < T \le T_\ast$, then there exists a unique triple $(u,p,\eta)$ that achieves the initial data, solves \eqref{ns_geometric}, and obeys the estimates 
\begin{equation}
 \sup_{0 \le t \le T} \left( \se_{n}^\sigma(t) + \f_{n}(t)\right) + \int_0^T \sd_{n}^\sigma(t) dt +  \ns{\dt^{n+1} u}_{(\mathcal{X}_T)^\ast}
\ls  \ns{u_0}_{2n} + \ns{\eta_0}_{2n+1/2} + \sigma \ns{\nab \eta_0}_{2n}.
\end{equation}
\end{thm}

\begin{remark}
 The functional framework of Theorem \ref{lwp} is sufficient to justify all of the a priori estimates we develop in this paper.
\end{remark}

\begin{remark}
 The compatibility conditions of Theorem \ref{lwp} allow us to construct the initial data for $\dt^j u(\cdot,0)$ and $\dt^j \eta(\cdot,0)$ for $j=1,\dotsc,n$ as well as the data $\dt^j p(\cdot,0)$ for $j=0,\dotsc,n-1$.  As such it makes sense to discuss $\se_{n}^\sigma(0)$ and $\f_{n}(0)$.
\end{remark}

\subsection{Statement of main results }

We now state our main results on the existence of global solutions to \eqref{ns_geometric}.  Our first result establishes the global well-posedness of the problem for a fixed positive value of surface tension as well as shows that these solutions decay at an exponential rate.  In this context we can work in a functional framework determined by $\se_{2}^\sigma$ and $\sd_{2}^\sigma$, i.e. with $n =2$.

\begin{thm}\label{gwp_st}
Fix $\sigma >0$.  Suppose that the initial data $(u_0,\eta_0)$ satisfy $\se_{2}^\sigma(0) < \infty$ as well as the compatibility conditions of Theorem \ref{lwp}.  There exist constants $\gamma_0 = \gamma_0(\sigma)\in (0,1)$ and $\kappa_0 = \kappa_0(\sigma) \in (0,1)$ such that if $\se_{2}^\sigma(0) \le \kappa_0$ and  $0 \le \gamma \le \gamma_0$, then there exists a unique triple $(u,p,\eta)$ that solves \eqref{ns_geometric} on the temporal interval $(0,\infty)$, achieves the initial data, and obeys the estimate
\begin{equation}
 \sup_{0 \le t \le \infty} e^{\lambda t} \se_{2}^\sigma(t) + \int_0^\infty\sd_{2}^\sigma(t) dt \ls \cs \se_{2}^\sigma(0)
\end{equation}
for a constant $\lambda = \lambda(\sigma,\gamma) >0$, where $\se_{2}^\sigma$ and $\sd_{2}^\sigma$ are as defined in \eqref{E_def} and \eqref{D_def}.
\end{thm}

Note that in this theorem the condition on $\f_{2}(0)$ used in the local theory has been removed.  This is possible because when $\sigma >0$ we have that $\f_{n}$ can be controlled by $\se_{n}^\sigma$.  See Proposition \ref{kf_ests_st} for a more precise statement.  Theorem \ref{gwp_st} theorem can be interpreted as saying that the trivial equilibrium $u=0$, $p=0$, $\eta=0$ is asymptotically stable for the problem \eqref{ns_geometric}.  The regularity of solutions in \eqref{gwp_st} is sufficiently high that we can change back to the Eulerian coordinates to produce global-in-time decaying solutions to \eqref{ns_euler} as well.  Thus we find that the steady shear solution \eqref{shear_soln} is asymptotically stable with an exponential rate of decay to equilibrium.

\begin{remark}
 Setting $\gamma =0$ in Theorem \ref{gwp_st} corresponds to the global well-posedness and exponential decay of solutions for the periodic viscous surface wave problem with surface tension.  This result was first established by Nishida-Teramoto-Yoshihara \cite{nis_ter_yos} using different techniques.  
\end{remark}

Theorem \ref{gwp_st} requires a fixed positive value of surface tension.  Our next main result considers the cases $\sigma =0$ and $\sigma$ small but positive.  We view the latter as the ``vanishing surface tension'' regime, as we will employ it to establish this limit.  In these cases we work in a more complicated functional setting that changes depending on whether $\sigma$ vanishes or not.  We introduce this with the following functional, defined for any integer $N \ge 3$ and time $t \in [0,\infty]$:
\begin{equation}\label{G_def}
 \g_{2N}^\sigma(t) = \sup_{0 \le r \le t} \se_{2N}^\sigma(r) + \int_0^t \sd_{2N}^\sigma(r)dr + \sup_{0 \le r \le t} (1+r)^{4N-8} \se_{N+2}^\sigma(r) + \sup_{0 \le r \le t} \frac{\f_{2N}(r)}{1+r}
\end{equation}
where here $\se_{n}^\sigma$, $\sd_{n}^\sigma$, and $\f_{n}$ are as defined by \eqref{E_def}, \eqref{D_def}, and \eqref{transport_def}, respectively.  Note that the condition $N \ge 3$ implies that $2N > N+2$ and that $4N-8 >0$.

We can now state our second main result.

\begin{thm}\label{gwp_vanish}
Let $N \ge 3$ and define $\g_{2N}^\sigma$ via \eqref{G_def}.  Suppose that the initial data $(u_0,\eta_0)$ satisfy $\se_{2N}^\sigma(0) + \f_{2N}(0) < \infty$ as well as the compatibility conditions of Theorem \ref{lwp}.  There exist universal constants $\gamma_0,\kappa_0 \in (0,1)$ such that if $\se_{2N}^\sigma(0) + \f_{2N}(0) \le \kappa_0$, $0 \le \gamma \le \gamma_0$, and $0 \le \sigma \le 1$, then there exists a unique triple $(u,p,\eta)$ that solves \eqref{ns_geometric} on the temporal interval $(0,\infty)$, achieves the initial data, and obeys the estimate
\begin{equation}\label{gwp_vanish_0}
 \g_{2N}^\sigma(\infty) \ls \se_{2N}^\sigma(0) + \f_{2N}(0).
\end{equation}
\end{thm}

In particular, the bound \eqref{gwp_vanish_0} establishes the decay estimate 
\begin{equation}
 \se_{N+2}^\sigma(t) \ls \frac{\se_{2N}^\sigma(0) + \f_{2N}(0)}{(1+t)^{4N-8}}.
\end{equation}
This is an algebraic decay rate, slower than the exponential rate proved in Theorem \ref{gwp_st}  with a fixed $\sigma >0$.  Two remarks about this are in order.  First, by choosing $N$ larger, we arrive a faster rate of decay. In fact, by taking $N$ to be arbitrarily large we can achieve arbitrarily fast algebraic decay rates, which is what is known as ``almost exponential decay.''  Of course, the trade-off in the theorem is that faster decay requires smaller data in higher regularity classes.  The second point is that when $0 < \sigma \le 1$ in the theorem, it is still possible to prove that $\se_{2N}^\sigma$ decays exponentially by modifying the arguments used later in Theorem \ref{apriori_st}.  We neglect to state this properly here because we only care about the vanishing surface tension limit, and in this case we cannot get uniform control of the exponential decay parameter $\lambda(\sigma,\gamma)$ from Theorem \ref{gwp_st}. 

Theorem \ref{gwp_vanish} also guarantees enough regularity to switch back to Eulerian coordinates.  Consequently, the theorem tells us that the shear solution \eqref{shear_soln} remains asymptotically stable without surface tension, but that the rate of decay to equilibrium is slower.

Our third result establishes the vanishing surface tension limit for the problem \eqref{ns_euler}.

\begin{thm}\label{vanishing_st}
Let $N \ge 3$ and consider a decreasing sequence $\{\sigma_m\}_{m=0}^\infty \subset (0,1)$ such that $\sigma_m \to 0$ as $m \to \infty$.  Let $\kappa_0,\gamma_0 \in (0,1)$ be as in Theorem \ref{gwp_vanish}, and assume that $0 \le \gamma \le \gamma_0$.  Suppose that for each $m \in \mathbb{N}$ we have initial data $(u_0^{(m)},\eta_0^{(m)})$ satisfying $\se_{2N}^{\sigma_m}(0) + \f_{2N}(0) < \kappa_0$ as well as the compatibility conditions of Theorem \ref{lwp}.  Let $(u^{(m)}, p^{(m)},\eta^{(m)})$ be the global solutions to \eqref{ns_geometric} associated to the data given by Theorem \ref{gwp_vanish}.  Further assume that 
\begin{equation}
u_0^{(m)}  \to u_0 \text{ in }H^{4N}(\Omega),  \;
 \eta_0^{(m)} \to \eta_0 \text{ in } H^{4N+1/2}(\Sigma),  \text{ and } \sqrt{\sigma_m} \nab \eta_0^{(m)} \to 0 \text{ in }H^{4N}(\Sigma)
\end{equation}
as $m \to \infty$.

Then the following hold.
\begin{enumerate}
 \item The pair $(u_0,\eta_0)$ satisfy the compatibility conditions of Theorem \ref{lwp} with $\sigma =0$.
 \item As $m \to \infty$, the triple  $(u^{(m)}, p^{(m)},\eta^{(m)})$ converges to  $(u, p,\eta)$, where the latter triple is the unique solution to \eqref{ns_geometric} with $\sigma=0$ and initial data $(u_0,\eta_0)$.  The convergence occurs in any space into which the space of triples $(u,p, \eta)$ obeying $\mathcal{G}^0_{2N}(\infty) < \infty$ compactly embeds.
\end{enumerate}

\end{thm}

\begin{remark}
Theorem \ref{vanishing_st} is modeled on similar results proved by Tan-Wang \cite{tan_wang} for the incompressible viscous surface wave problem.  Indeed, if we set $\gamma=0$ we recover their result.  A similar result for the compressible viscous surface-internal wave problem was established by Jang-Tice-Wang in \cite{jang_tice_wang_gwp}. 
\end{remark}

\subsection{Discussion and plan of paper}

The main focus of the paper is to establish a priori estimates for solutions to \eqref{ns_geometric}.  Indeed, once these are developed, the global existence and vanishing surface tension limit results of Theorems \ref{gwp_st}, \ref{gwp_vanish}, and \ref{vanishing_st} follow from a standard coupling to the local existence theory of Theorem \ref{lwp}.  We develop the scheme of a priori estimates by using a variant of the nonlinear energy method we employed with Guo \cite{guo_tice_per} to study the periodic viscous surface wave problem without surface tension.  The main steps of this scheme are as follows.

\textbf{Horizontal energy estimates: }

The main workhorse in our analysis is the energy-dissipation equation \eqref{nat_en_diss} and its linearized counterparts.  The natural physical energy and dissipation in \eqref{nat_en_diss} do no provide nearly enough control to close a scheme of a priori estimates, so we are forced to seek control of higher order derivatives by applying derivatives to the equation and again appealing to the energy-dissipation equation.  For this to work the differential operators we apply must be compatible with the boundary conditions in \eqref{ns_geometric}, and this restricts us ``horizontal'' derivatives of the form $\dt^{\alpha_0} \p_1^{\alpha_1} \p_2^{\alpha_2}$ for $\alpha \in \mathbb{N}^{1+2}$.

It turns out to be convenient to do these estimates in two forms, one for the equations written in ``geometric'' form as in \eqref{linear_geometric}, and the second written in ``flattened form'' as in \eqref{linear_flattened}.  The geometric form is better suited for analysis of the highest-order temporal derivatives, as it circumvents the problem of estimating the highest time derivatives of the pressure, which are not controlled in our scheme of estimates.  The flattened form works well for other derivatives and is convenient to use due to its compatibility with related elliptic estimates and due to the somewhat simpler, constant coefficient, form.  We develop these forms of the energy-dissipation equation in Section \ref{sec_energy_evolution}.

Summing the various energy-dissipation equations over an appropriate range of derivatives provides us with an equation roughly of the form
\begin{equation}\label{discussion_1}
 \frac{d}{dt} \seb_{n}^\sigma + \sdb_{n} = \gamma \mathcal{J}_n + \mathcal{I}_{n}^\sigma
\end{equation}
where $\mathcal{I}_n^\sigma$ is a cubic (or higher) interaction energy and $\mathcal{J}_n$ is a quadratic term generated by the shear flow background.  The form of $\mathcal{I}_n^\sigma$ grows more complicated with $n$  due to the fact that as we apply more derivatives to \eqref{ns_geometric} we introduce more commutators.  The more precise statement of \eqref{discussion_1} and its proof can be found in the first part of Section \ref{sec_apriori}.

It is important to note that the interaction term $\mathcal{I}_n^\sigma$ involves more differential operators than are controlled by either $\seb_{n}^\sigma$ or $\sdb_{n}$, so a nonlinear energy method based solely on the horizontal terms is impossible.  We are thus compelled to appeal to auxiliary estimates in order to gain control of more terms.

\textbf{Energy and dissipation enhancement: }

The next step in the nonlinear energy method is to employ various auxiliary estimates in order to gain control of more quantities in terms of those already controlled by $\seb_{n}^\sigma$ and $\sdb_{n}^\sigma$.  In other words, we seek to prove (again, roughly) that we have the comparison estimates
\begin{equation}\label{discussion_2}
 \se_{n}^\sigma \ls \seb_{n}^\sigma \le \se_{n}^\sigma \text{ and }  \sd_{n}^\sigma \ls \sdb_{n} \le \sd_{n}^\sigma.
\end{equation}
The main mechanisms for proving \eqref{discussion_2} are elliptic regularity for the Stokes problem and elliptic regularity for the capillary problem, both of which are recorded in Appendix \ref{app_elliptics}.  The elliptic estimates are coupled with delicate iteration arguments and a careful exploitation of the structure of the equations in \eqref{ns_geometric} in order to prove that \eqref{discussion_2} holds up to some error terms.  The proof of this is carried out in the latter part of Section \ref{sec_apriori}.

\textbf{Nonlinear estimates:}

In Section \ref{sec_nlins} we record the estimates of the various nonlinear terms that appear in the energy-dissipation, elliptic, and auxiliary estimates employed in the analysis.  A good portion of the nonlinearities can be handled in the usual way with a combination of product estimates, embeddings, and trace estimates.  However, a few of the nonlinearities present key challenges and must be treated delicately in order to arrive at a useful estimate.  

The nonlinear energy method requires more than just control of the nonlinearities: it requires structured control.  The rough idea here is that the nonlinear terms must be able to be absorbed by the dissipation functional in a small energy context.  For example, it is not enough to bound the term $\mathcal{I}_n^\sigma$ mentioned above via $\abs{\mathcal{I}_n^\sigma} \ls (\sd_{n}^\sigma)^r$ for some $r >1$.  We must instead have an estimate that is structured in a way compatible with absorption, i.e. one of the form 
\begin{equation}\label{discussion_3.1}
\abs{\mathcal{I}_n^\sigma} \ls (\se_n^\sigma)^{r} \sd_{n}^\sigma 
\end{equation}
for some $r>0$.  With such an estimate in hand we can work in a small-energy context, i.e. in the context of $\se_{n}^\sigma \ll 1$, in order to view $\abs{\mathcal{I}_n^\sigma}$ as a small multiple of the dissipation.

The shear interaction term $\mathcal{J}_n$ is less delicate in terms of its structure, but it is only quadratic.  Consequently, the best we can hope to prove is that 
\begin{equation}\label{discussion_3.2}
 \abs{\mathcal{J}_n }\ls \sd_{n}^\sigma,
\end{equation}
and with this in hand we can use the smallness of $\gamma$ to absorb the product $\gamma \mathcal{J}_n$.  Fortunately,  \eqref{discussion_3.2} holds and so this strategy is feasible.

\textbf{A priori estimates with surface tension:}

The previous three components combine to form a closed system of a priori estimates in the case $\sigma >0$.  Indeed, by combining \eqref{discussion_1}, the dissipation comparison in \eqref{discussion_2},  \eqref{discussion_3.1}, and \eqref{discussion_3.2} we roughly have that
\begin{equation}
 \frac{d}{dt} \seb_{n}^\sigma + \sd_{n}^\sigma \ls \gamma \sd_n^\sigma + (\se_n^\sigma)^r \sd_n^\sigma,
\end{equation}
and so if and $\gamma \ll 1$ and $\se_n^\sigma \ll 1$ on some temporal interval $[0,T]$ then we can deduce that 
\begin{equation}\label{discussion_4}
 \frac{d}{dt} \seb_{n}^\sigma + \hal \sd_{n}^\sigma \le 0.
\end{equation}
When $\sigma >0$ it is not hard to verify that the dissipation is coercive over the energy, i.e. $\se_{n}^\sigma \ls \sd_{n}^\sigma$, and so \eqref{discussion_4} and the energy comparison in \eqref{discussion_2} can be combined to show that
\begin{equation}\label{discussion_5}
 \sup_{0 \le t \le T} e^{\lambda t} \se_{n}^\sigma(t) + \int_0^T \sd_{n}^\sigma(t) dt \ls \se_{n}^\sigma(0)
\end{equation}
for some $\lambda >0$.  This shows that under a universal smallness condition on $\gamma$ and $\se_{n}^\sigma$ (which can be verified by the local theory) we have the stronger a priori estimate \eqref{discussion_5} with bound given in terms of the initial data.  The full details of this argument are found in Section \ref{sec_st_gwp}.

\textbf{A priori estimates with zero or vanishing surface tension:}

When $\sigma =0$ and in the vanishing surface tension analysis we cannot exploit the regularity gains afforded by the elliptic capillary problem.  This creates two serious problems.  The first is that without this control the dissipation fails to be coercive over the energy precisely due to a half-derivative gap in the estimate for $\eta$.  This means that we can no longer expect exponential decay of solutions.  The second and more severe problem is that the nonlinear estimates require control of $2n+1/2$ derivatives of $\eta$, whereas the energy only controls $2n$ derivatives of $\eta$.  This disparity is potentially disastrous even for the local existence theory, as it suggests derivative loss.  Fortunately, the kinematic transport equation for $\eta$ in \eqref{ns_geometric} provides an alternate way of estimating these derivatives and shows that they are finite.  Unfortunately, the best estimates associated to the transport equation give rise to bounds that grow linearly in time, and this poses serious problems for a nonlinear energy method in which the nonlinearity is supposed to be small in some sense.  

We get around these problems by employing the two-tier nonlinear energy method we developed with Guo in \cite{guo_tice_per} to handle the problem with $\gamma =0$.  The idea is to let $N \ge 3$ and consider together the high-order energy and dissipation $\se_{2N}^0$ and $\sd_{2N}^0$ along with the low-order energy and dissipation $\se_{N+2}^0$ and $\sd_{N+2}^0$.  We also use the functional $\f_{2N}$ defined by \eqref{transport_def} to track the highest derivatives of $\eta$.  

We control $\f_{2N}$ with a transport estimate in the first part of Section \ref{sec_nost_gwp}, but the estimate allows for $\f_{2N}$ to grow linearly in time.  To compensate for this in the nonlinear estimates of Section \ref{sec_nlins} we show that $\f_{2N}$ only appears in products with the very low regularity functional $\k$ defined by \eqref{K_def}.  We have a trivial estimate $\k \ls \se_{N+2}^0$, and so if we know a priori that $\se_{N+2}^0$ decays algebraically at a fast enough rate, then the product $\f_{2N} \k$ can be controlled uniformly in time.  Then, under the assumptions that $\g_{2N}^0 \ll 1$ and  $\gamma \ll 1$ we prove that (again, roughly)
\begin{equation}
 \sup_{0\le t \le T} \se_{2N}^0(t) + \int_0^T \sd_{2N}^0(t) dt \ls \se_{2N}^0(0).
\end{equation}
This means that decay of the low-tier energy allows us to close the high-tier bounds.

Next we use the high-tier bounds to show that the low-tier energy decays algebraically, with bounds in terms of the data.  The key point here is that $\sd_{N+2}^0$ is not coercive over $\se_{N+2}^0$, but it is possible to interpolate with the high-energy bound:
\begin{equation}
 \se_{N+2}^0 \ls (\sd_{N+2}^0)^\theta (\se_{2N}^0)^{1-\theta}
\end{equation}
for some $\theta = \theta(N) \simeq 1$.  This then allows us to prove a bound of the form
\begin{equation}
 \frac{d}{dt} \seb_{N+2} + C (\seb_{N+2})^{1 + 1/r} \le 0
\end{equation}
for some $r= r(N) >0$, and from this we can deduce the decay estimate
\begin{equation}
\se_{N+2}(t) \ls \se_{2N}(0) (1+t)^{-r}.
\end{equation}
This means that boundedness of the high-tier energy allows us to close the low-tier decay bounds.

The second part of Section \ref{sec_nost_gwp} contains the full details of the two-tier method and establishes the global well-posedness and algebraic decay of solutions.  An interesting feature of the two-tier analysis is that the existence of global solutions is predicated on their decay.

\textbf{Vanishing surface tension limit:}

Throughout the paper we take great care to isolate the behavior of constants with respect to $\sigma$.  This is done in order to allow us to send $\sigma \to 0$.  We carry our this analysis in the final part of Section \ref{sec_nost_gwp}.

\section{Evolution of the energy and dissipation}\label{sec_energy_evolution}

In this section we record the energy-dissipation evolution equations for two linearized versions of the problem \eqref{ns_geometric}: the geometric form and the flattened form.  We also record the forms of the nonlinear forcing terms that appear in the analysis of \eqref{ns_geometric}.

\subsection{Geometric form}

We assume that $u$ and $\eta$ are given and that $\Phi,\a,\n, J$, etc. are given in terms of $\eta$ as in \eqref{mapping_def} and \eqref{ABJ_def}.  The linearized geometric form of \eqref{ns_geometric} is then:
\begin{equation}\label{linear_geometric}
\begin{cases}
 \dt v - \dt \bar{\eta} \tilde{b} K \p_3 v + u \cdot \naba v + U \circ \Phi \cdot \naba v + v \cdot \naba (U \circ \Phi) + \diva S_{\a}(q,v) =F^1 & \text{in }\Omega \\
 \diva v =F^2 & \text{in } \Omega \\
 S_\a(q,v) \n = [(\zeta - \sigma \Delta \zeta) I - \gamma \zeta M  ]\n + F^3& \text{on } \Sigma \\
 \dt \zeta - v \cdot \n - s(\eta) \p_1 \zeta = F^4 & \text{on }\Sigma \\
 u = 0 & \text{on } \Sigma_b.
\end{cases}
\end{equation}

The next result records the energy-dissipation equation associated to solutions of \eqref{linear_geometric}.

\begin{prop}\label{lin_evolve_geo}
Let $\eta$ and $u$ be given and satisfy 
\begin{equation}\label{lin_evolve_geo_00}
\begin{cases}
\diva u =0 & \text{in }\Omega \\
\dt \eta + s(\eta) \p_1 \eta = u \cdot \n & \text{on }\Sigma.
\end{cases}
\end{equation}
Suppose that $(v,q,\zeta)$ solve \eqref{linear_geometric}, where $\Phi,\a,J$, etc are determined by $\eta$ as in \eqref{mapping_def} and \eqref{ABJ_def}.  Then 
\begin{multline}\label{lin_evolve_geo_01}
\frac{d}{dt} \left[\int_\Omega \hal \abs{v}^2 J + \int_{\Sigma} \frac{1}{2} \abs{\zeta}^2 + \frac{\sigma}{2} \abs{\nab \zeta}^2   \right] + \int_{\Omega} \hal \abs{\sg_{\a} v}^2 J =  \int_\Omega \gamma \Phi_3 v_3 v_1 J  + \int_\Sigma \gamma \zeta (M \n) \cdot v \\
+ \int_\Omega  J(v \cdot F^1 + q F^2) + \int_\Sigma -F^3 \cdot v +( \zeta - \sigma \Delta \zeta) \left(\frac{\gamma}{2}\eta^2 \p_1 \zeta + F^4  \right)
\end{multline}
\end{prop}
\begin{proof}
We take the dot product of the first equation in \eqref{linear_geometric} with $v$, multiply by $J$, and integrate over $\Omega$ to see that
 \begin{equation}\label{lin_evolve_geo_1}
 I + II = III
\end{equation}
for 
\begin{equation}
 I = \int_\Omega \dt v \cdot v J   - \dt \bar{\eta} \tilde{b} \p_3 v \cdot v   + (u\cdot \naba v)\cdot v J,
\end{equation}
\begin{equation}
 II = \int_\Omega  \diva S_{\a}(q,v)\cdot vJ, 
\text{ and }
 III = \int_\Omega F^1 \cdot v J - \int_\Omega [U \circ \Phi \cdot \naba v + v \cdot \naba (U \circ \Phi)] \cdot v J  
\end{equation}
In order to integrate by parts these terms we will utilize the geometric identity $\p_k(J \mathcal{A}_{ik})=0$ for each $i$, which is readily verified through direct computation.

To handle the term $I$ we first compute
\begin{equation}
 I = \frac{d}{dt}  \int_\Omega \frac{\abs{v}^2 J}{2} + \int_\Omega -\frac{\abs{v}^2 \dt J}{2} - \dt \bar{\eta} \tilde{b} \p_3 \frac{ \abs{v}^2}{2} +  u_j \p_k \left( J \mathcal{A}_{jk}  \frac{\abs{v}^2}{2} \right):= I_1 + I_2.
\end{equation}
Since $\tilde{b} = 1 + x_3/b$, an integration by parts and an application of the boundary condition $v=0$ on $\Sigma_b$ reveals that
\begin{multline}
I_2 = \int_\Omega -\frac{\abs{v}^2 \dt J}{2} - \dt \bar{\eta} \tilde{b} \p_3 \frac{ \abs{v}^2}{2} +  u_j \p_k \left( J \mathcal{A}_{jk}  \frac{\abs{v}^2}{2} \right) = \int_\Omega -\frac{\abs{v}^2 \dt J}{2} + \frac{\abs{v}^2}{2} \left( \frac{\dt \bar{\eta}}{b} + \tilde{b} \dt \p_3 \bar{\eta} \right) \\
- \int_\Omega J   \frac{\abs{v}^2}{2} \diva u + \hal \int_\Sigma - \dt \eta \abs{v}^2
+   u_j J \mathcal{A}_{jk} e_3\cdot e_k \abs{v}^2.
\end{multline}
Simple calculations show that $\dt J  = \dt \bar{\eta}/b + \tilde{b} \dt \p_3 \bar{\eta}$ in $\Omega$ and that  $J \mathcal{A}_{jk} e_3\cdot e_k = \n_j$ on $\Sigma$.  Combining these with the equations in \eqref{lin_evolve_geo_00} then shows that 
\begin{equation}
 I_2 = \int_\Sigma \frac{\abs{v}^2}{2} s(\eta) \p_1 \eta,
\end{equation}
and hence
\begin{equation}\label{lin_evolve_geo_2}
 I = I_1 + I_2 = \dt  \int_\Omega \frac{\abs{v}^2 J}{2} + \int_\Sigma \frac{\abs{v}^2}{2} s(\eta) \p_1 \eta.
\end{equation}

We begin our analysis of the term $II$ with a similar integration by parts, which reveals that
\begin{equation}
 II =  \int_\Omega -  S_\a(v,q) :\naba v J     + \int_\Sigma J \mathcal{A}_{j3} [S_{\a}(v,q)]_{ij} v_i 
= \int_\Omega -q \diva v J + J\frac{\abs{\sg_\mathcal{A} v}^2}{2} + \int_\Sigma  S_{\a}(v,q) \n \cdot v.
\end{equation}
We rewrite the integral on $\Sigma$ as 
\begin{multline}
\int_\Sigma  S_{\a}(v,q) \n \cdot v = \int_\Sigma (\zeta - \sigma \Delta \zeta) v\cdot \n - \gamma \zeta(M \n) \cdot v + \int_\Sigma F^3 \cdot v \\
= \int_\Sigma ( \zeta - \sigma \Delta \zeta)\left(\dt \zeta +\frac{\gamma}{2} \p_1 \zeta - \frac{\gamma}{2} \eta^2 \p_1 \zeta -F^4 \right) + \int_\Sigma F^3 \cdot v, 
\end{multline}
and then we compute
\begin{equation}
 \int_\Sigma ( \zeta - \sigma \Delta \zeta) \dt \zeta = \frac{d}{dt} \int_\Sigma \frac{1}{2} \abs{\zeta}^2 + \frac{\sigma}{2} \abs{\nab \zeta}^2 
\end{equation}
and 
\begin{equation}
 \int_\Sigma ( \zeta - \sigma \Delta \zeta) \frac{\gamma}{2} \p_1 \zeta =   \int_\Sigma \frac{1}{2}\p_1 \abs{\zeta}^2 + \frac{\sigma}{2} \p_1 \abs{\nab \zeta}^2 =0. 
\end{equation}
Combining these then shows that 
\begin{equation}\label{lin_evolve_geo_3}
 II = \int_\Omega -q \diva v J + J\frac{\abs{\sg_\mathcal{A} v}^2}{2}  + \frac{d}{dt} \int_\Sigma \frac{1}{2} \abs{\zeta}^2 + \frac{\sigma}{2} \abs{\nab \zeta}^2  - \int_\Sigma (\zeta-\sigma \Delta \zeta) \left( \frac{\gamma}{2} \eta^2 \p_1 \zeta +F^4 \right)
\end{equation}

We now rewrite the term $III$.  We begin by writing  
\begin{multline}
\int_\Omega   [U \circ \Phi \cdot \naba v ] \cdot v J = \int_\Omega J (U\circ \Phi)_j \a_{jk} \p_k\frac{\abs{v}^2}{2}   = \int_\Omega  (U\circ \Phi)_j \p_k\left( J\a_{jk} \frac{\abs{v}^2}{2}\right)   \\
= \int_\Omega - \p_k (U \circ \Phi)_i \a_{jk} J \frac{\abs{v}^2}{2} + \int_\Sigma U \circ \Phi \cdot \n \frac{\abs{v}^2}{2} = \int_\Omega - \diva (U\circ \Phi) \frac{\abs{v}^2}{2} J + \int_\Sigma s(\eta) (e_1 \cdot \n) \frac{\abs{v}^2}{2}.
\end{multline}
Note that 
\begin{multline}
 \diva (U\circ \Phi) = \a_{ij} \p_j (U\circ \Phi)_i = \a_{1j} \p_j (s \circ \Phi_3) = \a_{1j} s'\circ \Phi_3 \p_j \Phi_3 = s'\circ \Phi_3 (\nab \Phi)_{3j} \a^T_{j1} \\
 = s'\circ \Phi_3 (\nab \Phi \a^T)_{31} = s'\circ \Phi_3 I_{31} = 0
\end{multline}
and 
\begin{equation}
 -s(\eta) e_1 \cdot \n = s(\eta) \p_1 \eta.
\end{equation}
We may then combine the above to rewrite 
\begin{equation}\label{lin_evolve_geo_4}
 \int_\Omega   [U \circ \Phi \cdot \naba v ] \cdot v J = - \int_\Sigma s(\eta) \p_1 \eta \frac{\abs{v}^2}{2}.
\end{equation}
Next we compute 
\begin{multline}
(v \cdot \naba (U\circ \Phi))_i = v_j a_{jk} \p_k (U\circ \Phi)_i = \delta_{i1} v_j \a_{jk} \p_k (s \circ \Phi_3) = \delta_{i1} s'\circ \Phi_3  v_j \a_{jk} \p_k \Phi_3 \\
= \delta_{i1} s'\circ \Phi_3 v_j (\nab \Phi \a^T)_{3j} = \delta_{i1} s'\circ \Phi_3 v_j \delta_{3j} = \delta_{i1} s'\circ \Phi_3 v_3.
\end{multline}
This allows us to rewrite 
\begin{equation}\label{lin_evolve_geo_5}
 \int_\Omega   [v \cdot \naba (U\circ \Phi) ] \cdot v J = \int_\Omega -\gamma \Phi_3 v_3 v_1 J.
\end{equation}
We then combine \eqref{lin_evolve_geo_4} and \eqref{lin_evolve_geo_5} to see that 
\begin{equation}\label{lin_evolve_geo_6}
 III =     \int_\Omega v \cdot F^1 J + \int_\Sigma s(\eta) \p_1 \eta \frac{\abs{v}^2}{2} + \int_\Omega \gamma \Phi_3 v_3 v_1 J
\end{equation}

To conclude that  \eqref{lin_evolve_geo_01} holds we plug \eqref{lin_evolve_geo_2}, \eqref{lin_evolve_geo_3}, and \eqref{lin_evolve_geo_6} into \eqref{lin_evolve_geo_1}, rearrange, and cancel the term $\int_\Sigma s(\eta) \p_1 \eta \frac{\abs{v}^2}{2}$.

\end{proof}

Next we record the form of the forcing terms that will appear in our analysis.  We arrive at the forcing terms by applying $\dt^j$ to \eqref{ns_geometric}, so we will build the integer $j$ into our notation by writing $F^{i,j}$ for the $i^{th}$ forcing term generated by applying $\dt^j$.  

We have that the first term is $F^{1,j} = \hat{F}^{1,j} + \tilde{F}^{1,j},$ for 
\begin{equation}\label{nlin_F1}
\begin{split}
\hat{F}^{1,r}_{i} &:=  \sum_{0 < \ell \le r} C_{r\ell} \left[ \dt^\ell (\dt \bar{\eta} \tilde{b} K) \dt^{r-\ell} \p_{3} u_{i}  - \dt^\ell(u_{j} \mathcal{A}_{jk}) \dt^{r-\ell} \p_{k} u_{i} + \dt^\ell\mathcal{A}_{ik} \dt^{r-\ell}\p_{k}  p  \right],\\
&+  \sum_{0 < \ell \le r} C_{r\ell} \left[ \dt^\ell \mathcal{A}_{j k} \dt^{r-\ell} \p_{k} ( \mathcal{A}_{im} \p_{m} u_{j} + \mathcal{A}_{jm} \p_{m} u_{i})  + \mathcal{A}_{jk} \p_{k} (\dt^\ell \mathcal{A}_{i m} \dt^{r-\ell} \p_{m} u_{j} + \dt^\ell \mathcal{A}_{j m} \dt^{r-\ell} \p_{m} u_{i}) \right], \\
\tilde{F}^{1,r}_{i} &:= -\sum_{0 < \ell \le r} C_{r\ell} \left[ \dt^\ell(s \circ \Phi_3) \dt^{r-\ell} (\a_{1k} \p_k u_i) + s \circ \Phi_3  \dt^\ell \a_{ik} \dt^{r-\ell} \p_k u_i \right].
\end{split}
\end{equation}
The second is
\begin{equation}\label{nlin_F2}
F^{2,r}   :=    - \sum_{0 < \ell \le r} C_{r\ell} \dt^\ell \mathcal{A}_{ij} \dt^{r-\ell} \p_{j} u_{i}, 
\end{equation}
and the third is $F^{3,r} = \hat{F}^{3,r} + \tilde{F}^{3,r}$, for 
\begin{equation}\label{nlin_F3}
\begin{split}
\hat{F}^{3,r}_{i} &:=  \sum_{0 < \ell \le r} C_{r\ell} \left[\dt^{r-\ell}  (\eta - p) \dt^\ell \mathcal{N}_{i}
  +  \dt^{r-\ell} \left( \mathcal{A}_{ik} \p_{k}u_{j} + \mathcal{A}_{jk} \p_{k} u_{i}  \right) \dt^\ell \mathcal{N}_{j} \right] \\
&+  \sum_{0 < \ell \le r} C_{r\ell} \left[     \left( \dt^\ell \mathcal{A}_{ik} \dt^{r-\ell} \p_{k} u_{j} + \dt^\ell \mathcal{A}_{jk}\dt^{r-\ell} \p_{k} u_{i} \right) \mathcal{N}_{j}  \right],
  \\
\tilde{F}^{3,r}_i  &:= - \sum_{0 < \ell \le r} C_{r\ell} \left[ \dt^{r-\ell} (\sigma \Delta \eta I_{ik} + \gamma \eta M_{ik}) \dt^{\ell}\n_k \right] + \sum_{0 \le \ell \le r} C_{r\ell} [ \dt^\ell(\sigma \Delta \eta - \sigma \mathfrak{H}(\eta)) \dt^{r-\ell} \n_i  ]. 
\end{split}
\end{equation}
The fourth is
\begin{equation}\label{nlin_F4}
 F^{4} := \sum_{0 < \ell \le r} C_{r\ell} \left[\dt^{r-\ell} u \cdot \dt^\ell \n + \frac{\gamma}{2} \dt^\ell (\eta^2) \dt^{r-\ell} \p_1 \eta \right]  .
\end{equation}

\subsection{Flattened form}

It will also be useful for us to have a linearized version of \eqref{ns_geometric} in which the operators have constant coefficients.  This version is as follow:
\begin{equation}\label{linear_flattened}
\begin{cases}
 \dt v  + s \p_1 v + s' v_3 e_1 + \diverge S (q,v) = \Theta^1  & \text{in }\Omega \\
 \diverge v =\Theta^2 & \text{in } \Omega \\
 S(q,v)e_3 = [(\zeta - \sigma \Delta \zeta) I ]e_3 - \gamma \zeta e_1  + \Theta^3& \text{on } \Sigma \\
 \dt \zeta = v_3 - (\gamma/2) \p_1 \zeta + \Theta^4 & \text{on }\Sigma \\
 v = 0 & \text{on } \Sigma_b.
\end{cases}
\end{equation}

The next result records the energy-dissipation equation associated to solutions of \eqref{linear_flattened}.

\begin{prop}\label{lin_evolve_flat}
Suppose that $(v,q,\zeta)$ solve \eqref{linear_flattened}.  Then  
\begin{multline}\label{lin_evolve_flat_00}
\frac{d}{dt} \left[\int_\Omega \hal \abs{v}^2 + \int_{\Sigma} \frac{1}{2} \abs{\zeta}^2 + \frac{\sigma}{2} \abs{\nab \zeta}^2   \right] + \int_\Omega \frac{1}{2} \abs{\sg v}^2 = \int_\Omega \gamma x_3 v_3 v_1 + \int_\Sigma \gamma \zeta v_1 \\
+ \int_\Omega  v \cdot \Theta^1 + q \Theta^2 + \int_\Sigma -\Theta^3 \cdot v +( \zeta - \sigma \Delta \zeta) \Theta^4  
\end{multline}
\end{prop}
\begin{proof}
We take the dot product of the first equation in \eqref{linear_flattened} with $v$ and integrate over $\Omega$ to see that 
\begin{equation}\label{lin_evolve_flat_1}
 I + II = III
\end{equation}
for 
\begin{equation}
 I = \int_\Omega \dt v \cdot v, II = \int_\Omega \diverge S(q,v) \cdot v, III = \int_\Omega v \cdot \Theta^1 - v\cdot(s \p_1 v + s' v_3 e_1).
\end{equation}
We have that 
\begin{equation}\label{lin_evolve_flat_2}
 I = \frac{d}{dt} \int_\Omega \hal \abs{v}^2.
\end{equation}
Next we compute
\begin{equation}
II =  \int_\Omega - S(q,v) : \nab v + \int_\Sigma S(q,v)e_3 \cdot v := II_1 + II_2.
\end{equation}
A simple computation reveals that 
\begin{equation}
 II_1 = \int_\Omega \frac{1}{2} \abs{\sg v}^2 - q \Theta^2.
\end{equation}
We use \eqref{linear_flattened} to rewrite 
\begin{multline}
 II_2 = \int_\Sigma ( \zeta - \sigma \Delta \zeta) v_3 - \gamma \zeta v_1 + \Theta^3 \cdot v = \int_\Sigma ( \zeta - \sigma \Delta \zeta) \left(\dt \zeta  +\frac{\gamma}{2} \p_1 \zeta  -\Theta^4 \right) - \gamma \zeta v_1 + \Theta^3 \cdot v \\
= \frac{d}{dt} \left( \int_\Sigma \frac{1}{2}\abs{\zeta}^2 + \frac{\sigma}{2} \abs{\nab \zeta}^2 \right) -  \int_\Sigma ( \zeta - \sigma \Delta \zeta) \Theta^4 + \gamma \zeta v_1 - \Theta^3 \cdot v, 
\end{multline}
which shows that
\begin{equation}\label{lin_evolve_flat_3}
 II = \int_\Omega \frac{1}{2} \abs{\sg v}^2 - q \Theta^2 +  \frac{d}{dt} \left( \int_\Sigma \frac{1}{2} \abs{\zeta}^2 + \frac{\sigma}{2} \abs{\nab \zeta}^2 \right) -  \int_\Sigma ( \zeta - \sigma \Delta \zeta) \Theta^4 + \gamma \zeta v_1 - \Theta^3 \cdot v.
\end{equation}
Finally, we eliminate one of the terms in $III$ by using the fact that $s = s(x_3)$:
\begin{equation}\label{lin_evolve_flat_4}
 \int_\Omega v \cdot s \p_1 v = \int_\Omega  \p_1 \left(s\frac{\abs{v}^2}{2}\right) = 0.
\end{equation}
Then \eqref{lin_evolve_flat_00} follows by combining \eqref{lin_evolve_flat_1}, \eqref{lin_evolve_flat_2},  \eqref{lin_evolve_flat_3}, and \eqref{lin_evolve_flat_4}.

\end{proof}

Next we record the exact forms of the forcing terms that appear in \eqref{linear_flattened} when we rewrite \eqref{ns_geometric} in this form.  The first is $G^1 = \hat{G}^2 + \tilde{G}^1$ for
\begin{equation}\label{G_1_hat}
\hat{G}^1 =  - \diverge (\sg_{I-\a} u) - \diverge_{\a-I}(pI - \sg_\a u) - u\cdot \naba u  +  \dt \bar{\eta} \tilde{b} K \p_3 u
\end{equation}
and
\begin{equation}\label{G_1_tilde}
 \tilde{G}^1 = -(s \circ \Phi_3 - s) \p_1 u + (s\circ \Phi_3 e_1) \cdot \nab_{\a - I} u - u \cdot \nab_{\a-I} (s \circ \Phi_3 e_1) - u \cdot \nab [ (s\circ \Phi_3 - s)e_1  ].
\end{equation}
The second is
\begin{equation}\label{G_2_hat}
 G^2 = \diverge_{(I - \a)} u
\end{equation}
and the third is
$G^3 = \hat{G}^3 + \check{G}^3 +\tilde{G}^3$ for
\begin{equation}\label{G_3_hat}
 \hat{G}^3 = [- S_\a(p,u) + \eta I ](\n-e_3) - \sg_{(I-\a)} u e_3 
\end{equation}
\begin{equation}\label{G_1_check}
 \check{G}^3 =  \sigma(\Delta \eta - \mathfrak{H}(\eta)) e_3 + \sigma \mathfrak{H}(\eta)(e_3 - \n) 
\end{equation}
and
\begin{equation}\label{G_3_tilde}
 \tilde{G}^3 = \gamma \eta M(e_3 - \n).
\end{equation}
The fourth is $G^4 = \hat{G}^4 + \tilde{G}^4$ for 
\begin{equation}\label{G_4_hat}
 \hat{G}^4 = u \cdot(\n-e_3)
\end{equation}
and
\begin{equation}\label{G_4_tilde}
 \tilde{G}^4 = -\frac{\gamma}{2} \eta^2 \p_1 \eta.
\end{equation}

\section{Estimates of the nonlinearities }\label{sec_nlins}

In this section we develop the estimates of the nonlinearities needed to close our scheme of a priori estimates.  

\subsection{$L^\infty$ estimates}

The next result establishes some key $L^\infty$ bounds that will be used repeatedly throughout the paper.

\begin{prop}\label{infty_ests}
There exists a universal constant $\delta \in (0,1)$ such that if $\ns{\eta}_{5/2} \le \delta$, then the following bounds hold.
\begin{enumerate}
 \item We have that
\begin{equation}
 \norm{J-1}_{L^\infty} +\norm{\n-1}_{L^\infty} + \norm{A}_{L^\infty} + \norm{B}_{L^\infty} \le \hal 
\end{equation}
and 
\begin{equation}
 \norm{K}_{L^\infty} + \norm{\a}_{L^\infty} \ls 1.
\end{equation}
\item The mapping $\Phi$ given by \eqref{mapping_def} is a diffeomorphism from $\Omega$ to $\Omega(t)$.
\item For all $v \in H^1(\Omega)$ such that $v=0$ on $\Sigma_b$ we have the estimate 
\begin{equation}
 \int_\Omega \abs{\sg v}^2 \le \int_\Omega J \abs{\sg_\a v}^2 + C\left( \norm{\a-I}_{L^\infty} + \norm{J-1}_{L^\infty} \right) \int_\Omega \abs{\sg v}^2
\end{equation}
for a universal constant $C>0$.
\end{enumerate}

\end{prop}

\begin{proof}
The proof of the first item, based on Lemmas \ref{sobolev_product_1}, \ref{p_poisson}, and \ref{p_poisson_2}, may be found in Lemma 2.4 of \cite{guo_tice_per}.  The second follows easily from the embedding $H^{3}(\Omega) \hookrightarrow C^1(\Omega)$.  The proof of the third item can be found in proof of Proposition 4.3 in \cite{guo_tice_per}.
\end{proof}

\subsection{Estimates of the $G$ forcing terms}

We now present the estimates for the $G^i$ nonlinearities.  Estimates of the same general form are now well-known in the literature: \cite{guo_tice_per,jang_tice_wang_gwp,tan_wang}.

\begin{thm}\label{nlin_G}
Let $G^1,G^2,G^3,G^4$ be given by \eqref{G_1_hat}--\eqref{G_4_tilde}. Assume that $\se_{n}^\sigma \le \delta$ for the universal $\delta \in (0,1)$ given by Proposition \ref{infty_ests}.  Then there exists a polynomial $P(\cdot,\cdot)$ with non-negative universal coefficients such that
\begin{equation}\label{n_G_0}
\sum_{j=0}^{n-1} \ns{ \dt^j  G^1}_{2n-2j-2}  +  \ns{ \dt^j G^2}_{2n -2j - 1} +
 \ns{\dt^j  G^3}_{2n-2j -3/2} + \ns{\dt^j G^4}_{2n-2j-1/2 }
 \ls P(\sigma,\gamma) (\se_{n}^0 \se_{n}^\sigma + \k \f_{n}),
\end{equation}
and
\begin{multline}\label{n_G_00}
 \sum_{j=0}^{n-1}\left( \ns{ \dt^j  G^1}_{2n-2j-1}  +  \ns{ \dt^j G^2}_{2n -2j} +
 \ns{\dt^j  G^3}_{2n-2j -1/2} \right)  
+  \ns{ G^4}_{2n-1/2} + \sigma^2 \ns{G^4}_{2n+1/2}  \\
+ \ns{\dt G^4}_{2n-2} + \sigma^2 \ns{\dt G^4}_{2n-3/2} + \sum_{j=2}^{n} \ns{\dt^j  G^4}_{2n-2j+1/2}
\ls P(\sigma,\gamma) \left( \se_{n}^0  \sd_{n}^\sigma +  \k \f_{n}\right),
\end{multline}

\end{thm}

\begin{proof}

The definition of the $G^i$ terms in \eqref{G_1_hat}--\eqref{G_4_tilde} allows us to write  $G^i = \hat{G}^i + \check{G}^i + \tilde{G}^i$, where each of the sub-terms is either a nonlinearity of at least quadratic order, or else zero.  Consequently, it suffices to prove the bounds \eqref{n_G_0} and \eqref{n_G_00} with $G^i$ replaced by $\hat{G}^i$, $\check{G}^i$, and $\tilde{G}^i$.  Note that this notation has been employed so that only the terms $\check{G}^i$ involve the  parameter $\sigma$ and only the terms $\tilde{G}^i$ involve the parameter $\gamma$.

The proofs of the estimates for the $\hat{G}^i$ and $\check{G}^i$ terms can be found in Lemma 3.3 of \cite{jang_tice_wang_gwp}, though the numbering scheme for the nonlinearities is slightly different there.  Note, though, that the resulting estimates do not involve $\gamma$, and so the polynomial on the right side of the estimates do not depend on $\gamma$.   A simple modification of the arguments used in Lemma 3.3 of \cite{jang_tice_wang_gwp} yields the estimates for the $\tilde{G}^i$ terms as well, but in this case the right sides of the estimates have polynomials in $\gamma$.  For the sake of brevity we will omit further details.
\end{proof}

Our next result provides some bounds for nonlinearities appearing in integrals.

\begin{prop}\label{spatial_int_est}
Let $\alpha \in \mathbb{N}^2$ be such that $\abs{\alpha}=n$.  Assume that $\se_{n}^\sigma \le \delta$ for the universal $\delta \in (0,1)$ given by Proposition \ref{infty_ests}.  Then
\begin{equation}\label{eta es}
\abs{ \int_{\Sigma}   \pa \eta \pa  G^4 } \ls \sqrt{\se_{n}^0}  \sd_{n}^0 +\sqrt{  \sd_{n}^0 \k \f_{n}}
\end{equation}
and 
\begin{equation}\label{eta es2}
\abs{  \int_{\Sigma}   \sigma \Delta \pa \eta \pa  G^4  } \ls  \ks \left[ \sqrt{\se_{n}^0 \sd_{n}^0 \sd_{n}^\sigma}  +   \sqrt{  \sd_{n}^\sigma \k \f_{n}} \right] .
\end{equation}
\end{prop}
\begin{proof}
These estimates are proved with slightly different notation in Lemma 3.5 of \cite{jang_tice_wang_gwp}.  
\end{proof}

\subsection{Estimates of the $F$ forcing terms}

We now present the estimates of the $F$ forcing terms that appear in the geometric form of the equations \eqref{linear_geometric}.

\begin{thm}\label{nlin_F}
Let $F^{i,j}$ be defined by \eqref{nlin_F1}--\eqref{nlin_F4} for $0\le j\le n$. Assume that $\se_{n}^\sigma \le \delta$ for the universal $\delta \in (0,1)$ given by Proposition \ref{infty_ests}.  Then
\begin{equation}\label{nlin_F_0}
 \ns{F^{1,j} }_{0}+ \ns{  F^{2,j} }_{0}  +\ns{\dt(JF^{2,j})}_{0}  + \ns{F^{3,j}}_{0} + \ns{F^{4,j}}_{0} \ls P(\sigma,\gamma) \se_{n}^0 \sd_{n}^\sigma
\end{equation} 
and
\begin{equation}\label{nlin_F_00}
 \ns{F^{2,n}}_0 \ls (\se_{n}^0)^2.
\end{equation}
\end{thm}
\begin{proof}
Note that all terms in the definitions of $F^{i,j}$ are at least quadratic; each term can be written in the form $X Y$, where $X$ involves fewer temporal derivatives than $Y$. We may use the usual Sobolev embeddings, trace theory, and Proposition \ref{infty_ests} along with the definitions of $\se_{n}^\sigma$ and $\sd_{n}^\sigma$ to estimate $\norm{X}_{L^\infty}^2\ls  \se_{n}^0$ and $\norm{Y}_{0}^2\ls P(\sigma,\gamma) \sd_{n}^0$.  Then $\norm{XY}_0^2\le \norm{X}_{L^\infty}^2\norm{Y}_{0}^2\ls P(\sigma,\gamma) \se_{n}^0 \sd_{n}^\sigma$, and the estimate \eqref{nlin_F_0} follows by summing.  A similar argument proves \eqref{nlin_F_00}.
\end{proof}

Later in the paper we will encounter the functional
\begin{equation}\label{Hn_def}
 \mathcal{H}_n = \int_\Omega - \dt^{n-1} p  F^{2,n} J +  \hal \abs{\dt^n u}^2(J-1). 
\end{equation}
We will need to be able to estimate this in the following particular way.

\begin{prop}\label{Hn_estimate}
Let $\mathcal{H}_n$ be given by \eqref{Hn_def}, and suppose that $n \ge 2$. Assume that $\se_{n}^\sigma \le \delta$ for the universal $\delta \in (0,1)$ given by Proposition \ref{infty_ests}.  Then 
\begin{equation}
 \abs{\mathcal{H}_n} \ls (\se_{n}^0)^{3/2}
\end{equation}
\end{prop}
\begin{proof}
We bound
\begin{equation}
 \abs{\mathcal{H}_n} \le \norm{\dt^{n-1} p}_0 \norm{F^{2,n}}_0 \norm{J}_{L^\infty} + \hal \norm{J-1}_{L^\infty} \ns{\dt^n u}_0.
\end{equation}
Then we use Proposition \ref{infty_ests} and Theorem \ref{nlin_F} to estimate  
\begin{equation}
 \norm{F^{2,n}}_0 \norm{J}_{L^\infty} \ls \se_{n}^0
\end{equation}
and we use the Sobolev embeddings to bound
\begin{equation}
 \norm{J-1}_{L^\infty} \ls \norm{\bar{\eta}}_{C^1(\Omega)} \ls \norm{\bar{\eta}}_{H^3(\Omega)} \ls \norm{\eta}_{5/2} \ls \sqrt{\se_{n}^0} 
\end{equation}
since $n \ge 2$.  Consequently, 
\begin{equation}
 \abs{\mathcal{H}_n} \ls  \sqrt{\se_{n}^0} \left( \norm{\dt^{n-1}p}_0 \sqrt{\se_{n}^0} + \ns{\dt^n u}_0\right) \ls (\se_{n}^0)^{3/2},
\end{equation}
which is the desired estimate.
\end{proof}

\section{General a priori estimates }\label{sec_apriori}

The purpose of this section is to present a priori estimates that are general in the sense that they are valid for both the problem with and without surface tension.  The general estimates presented here will be specially adapted later to each problem to prove different sorts of results.

\subsection{Energy-dissipation evolution estimates  }

For $\alpha \in \mathbb{N}^{1+2}$ write 
\begin{equation}\label{sb_alpha_def}
 \seb_\alpha^\sigma = \int_\Omega \hal \abs{\pa u}^2 + \int_{\Sigma} \frac{1}{2} \abs{\pa \eta}^2 + \frac{\sigma}{2} \abs{\nab \pa \eta}^2
\text{ and } \sdb_\alpha = \int_\Omega \frac{1}{2} \abs{\sg \pa u}^2
\end{equation}

Our first result derives energy-dissipation estimates related to pure temporal derivatives of the highest order.

\begin{thm}\label{time_derivative_est}
Assume that $\se_{n}^\sigma \le \delta$ for the universal $\delta \in (0,1)$ given by Proposition \ref{infty_ests}.  Let $\alpha \in \mathbb{N}^{1+2}$ be given by $\alpha = (n,0,0)$, i.e. $\p^a = \dt^n$.  Then for $\seb_\alpha^\sigma$ and $\sdb_\alpha$  given by \eqref{sb_alpha_def} we have the estimate
\begin{equation}\label{td_0}
\frac{d}{dt} \left( \seb_\alpha^\sigma  + \mathcal{H}_n \right) + \sdb_\alpha  \ls \gamma \sd_{n}^0 + \ks \sqrt{\se_{n}^0} \sd_{n}^\sigma
\end{equation}
where $\mathcal{H}_N$ is given by \eqref{Hn_def}.
\end{thm}
\begin{proof}

We apply Proposition \ref{lin_evolve_geo} with $v = \dt^{n} u$, $q = \dt^{n} p$, and $\zeta = \dt^{n} \eta$ in order to see that 
\begin{multline}\label{td_1}
\frac{d}{dt} \left[\int_\Omega \hal \abs{\dt^{n} u}^2 J + \int_{\Sigma} \frac{1}{2} \abs{\dt^{n} \eta }^2 + \frac{\sigma}{2} \abs{\nab \dt^{n} \eta}^2   \right] + \int_{\Omega} \hal \abs{\sg_{\a} \dt^{n} u}^2 J \\
=  \int_\Omega \gamma \Phi_3 \dt^{n} u_3 \dt^{n} u_1 J  + \int_\Sigma \gamma \dt^{n} \eta (M \n) \cdot \dt^{n} u \\
+ \int_\Omega  J(\dt^{n} u \cdot F^{1,n} + \dt^{n} p F^{2,n}) + \int_\Sigma -F^{3,n} \cdot \dt^{n} u + ( \dt^{n} \eta - \sigma \Delta \dt^{n} \eta) \left(\frac{\gamma}{2}\eta^2 \p_1 \dt^{n}  \eta + F^{4,n}  \right)
\end{multline}
where $F^{i,n}$ are defined by \eqref{nlin_F1}--\eqref{nlin_F4}.  Our goal now is to estimate the terms on right side of \eqref{td_1} and then to rewrite some of the terms on the left.

We begin by estimating the terms on the right side of \eqref{td_1}.  We handle the first two terms  by using Proposition \ref{infty_ests} and trace theory to estimate 
\begin{equation}\label{td_2}
 \abs{\int_\Omega \gamma \Phi_3 \dt^{n} u_3 \dt^{n} u_1 J  + \int_\Sigma \gamma \dt^{n} \eta (M \n) \cdot \dt^{n} u} \ls \gamma \left( \ns{\dt^n u}_0 + \norm{\dt^n \eta}_0 \norm{\dt^n u}_1  \right) \ls \gamma \sd_{n}^0.
\end{equation}
To handle the pressure term we first rewrite
\begin{equation}\label{td_3}
 \int_\Omega   \dt^{n} p J F^{2,n} = \frac{d}{dt} \int_\Omega \dt^{n-1} p J F^{2,n} - \int_\Omega \dt^{n-1} p \dt(J F^{2,n}).
\end{equation}
We then use Theorem \ref{nlin_F} to estimate
\begin{equation}\label{td_4}
 \abs{\int_\Omega \dt^{n-1} p \dt(J F^{2,n})} \le \norm{\dt^{n-1} p}_0 \norm{\dt(J F^{2,n})}_{0} \ls \ks \sqrt{\se_{n}^0} \sd_{n}^\sigma.
\end{equation}
Next we use Theorem \ref{nlin_F}, Proposition \ref{infty_ests}, and trace theory allow us to estimate
\begin{multline}\label{td_5}
 \abs{ \int_\Omega  J\dt^{n} u \cdot F^{1,n} + \int_\Sigma -F^{3,n} \cdot \dt^{n} u} \ls \norm{\dt^{n} u}_1 \left( \norm{F^{1,n}}_0 +  \norm{F^{3,n}}_0 \right) \\
 \ls \ks \sqrt{\sd_{n}^0} \sqrt{\se_{n}^0 \sd_{n}^\sigma} \ls \ks \sqrt{\se_{n}^0} \sd_{n}^\sigma.
\end{multline}
We complete the analysis of the right side of \eqref{td_1} by again using Theorem \ref{nlin_F} and the usual Sobolev embeddings to bound
\begin{multline}\label{td_6}
 \abs{ \int_\Sigma ( \dt^{n} \eta - \sigma \Delta \dt^{n} \eta) \left(\frac{\gamma}{2}\eta^2 \p_1 \dt^{n}  \eta + F^{4,n}  \right) } \ls  (1+\sigma) \norm{\dt^{n} \eta}_2  \left( \gamma \norm{\eta}_{L^\infty}^2 \norm{\dt^{n} \eta}_1 + \norm{F^{4,n}}_0  \right) \\
\ls \ks \norm{\dt^{n} \eta}_2  \left(\norm{\eta}_{2}^2 \norm{\dt^{n} \eta}_1 + \norm{F^{4,n}}_0  \right)
 \ls \ks \sqrt{\sd_{n}^0} \left(\se_{n}^0 \sqrt{\sd_{n}^0} +  \sqrt{\se_{n}^0 \sd_{n}^\sigma} \right) \\
\ls \ks \sqrt{\se_{n}^0} \sd_{n}^\sigma.
\end{multline}

Next we rewrite some of the terms on the left side of \eqref{td_1}.  First, Proposition \ref{infty_ests} allows us to bound
\begin{equation}\label{td_7}
\hal \int_\Omega \abs{\sg \dt^n u}^2 \le  \int_{\Omega} \hal \abs{\sg_{\a} \dt^{n} u}^2 J + C \sqrt{\se_{n}^0} \sd_{n}^0.
\end{equation}
Second, we rewrite 
\begin{equation}\label{td_8}
 \int_\Omega \hal \abs{\dt^n u}^2 J = \int_\Omega \hal \abs{\dt^n u}^2 + \int_\Omega \hal \abs{\dt^n u}^2 (J-1)
\end{equation}

The estimate \eqref{td_0} now follows by combining \eqref{td_1}--\eqref{td_8}.
\end{proof}

Our next result provides energy-dissipation estimates for all derivatives besides the highest order temporal ones.

\begin{thm}\label{space_derivative_est}
Assume that $\se_{n}^\sigma \le \delta$ for the universal $\delta \in (0,1)$ given by Proposition \ref{infty_ests}.
    Let $\alpha \in \mathbb{N}^{1+2}$ be such that $\abs{\alpha} \le 2n$ and $\alpha_0 < n$.  Then for $\seb_\alpha^\sigma$ and $\sdb_\alpha$  given by \eqref{sb_alpha_def} we have the estimate
\begin{equation}\label{sd_0}
 \frac{d}{dt} \seb_\alpha^\sigma + \sdb_\alpha \ls \gamma \sd_{n}^0 + \ks \left[ \sqrt{\se_{n}^0 } \sd_{n}^\sigma  +   \sqrt{  \sd_{n}^\sigma \k \f_{n}} \right].
\end{equation}

\end{thm}

\begin{proof}

We begin by applying Proposition \ref{lin_evolve_flat} to see that 
\begin{multline}\label{sd_1}
 \frac{d}{dt} \seb_\alpha^\sigma + \sdb_{\alpha} = \int_\Omega \gamma x_3 \pa u_3 \pa u_1 + \int_\Sigma \gamma \pa\eta \pa u_1 
+ \int_\Omega  \pa u \cdot \pa G^1 + \pa p \pa G^2 \\
+ \int_\Sigma -\pa G^3 \cdot \pa u +( \pa \eta - \sigma \Delta \pa \eta) \pa G^4.  
\end{multline}
We will now estimate all of the terms appearing on the right side of \eqref{sd_1}.  We begin with the first two terms.  For these we use trace theory to bound
\begin{multline}\label{sd_1_2}
 \abs{\int_\Omega \gamma x_3 \pa u_3 \pa u_1 + \int_\Sigma \gamma \pa\eta \pa u_1} \ls \gamma \left( \ns{u}_{2n} + \norm{\eta}_{2n-1/2}\norm{u}_{H^{2n+1/2}(\Sigma)} \right) \\
 \ls \gamma \left( \ns{u}_{2n} + \norm{\eta}_{2n-1/2}\norm{u}_{2n+1} \right) \le \gamma \sd_{n}^0.
\end{multline}
In order to estimate the remaining terms on the right side of \eqref{sd_1} we will break to cases based on $\alpha$.

\emph{Case 1 -- Pure spatial derivatives of highest order}  

First assume that $\abs{\alpha}=2n$ and $\alpha_0 = 0$, i.e. $\pa$ is purely spatial derivatives of the highest order.  It will be convenient to write $\alpha = \beta + \delta$ for $\abs{\beta} =1$.  This then allows us to bound the $G^1$ term on the right side of \eqref{sd_1} by integrating by parts and employing \eqref{n_G_00} of Theorem \ref{nlin_G}:
\begin{equation}\label{sd_2}
\abs{ \int_\Omega  \pa u \cdot \pa G^1 } = \abs{\int_\Omega \p^{\beta + \alpha} u \cdot \p^\delta G^2 } \le \norm{u}_{2n+1} \norm{G^1}_{2n-1} \ls \ks \sqrt{\sd_{n}^0} \sqrt{\se_{n}^0 \sd_{n}^\sigma + \k \f_{n}}.
\end{equation}
We also use Theorem \ref{nlin_G} to estimate the $G^2$ and $G^3$ terms:
\begin{equation}\label{sd_3}
 \abs{\int_\Omega    \pa p \pa G^2} \le \norm{p}_{2n} \norm{G^2}_{2n} \ls \ks \sqrt{\sd_{n}^0} \sqrt{\se_{n}^0 \sd_{n}^\sigma + \k \f_{n}}
\end{equation}
and
\begin{equation}\label{sd_4}
 \abs{\int_\Sigma \pa G^3 \cdot \pa u} \le \norm{\pa u}_{H^{1/2}(\Sigma)} \norm{\pa G^3}_{H^{-1/2}(\Sigma)} \ls \norm{\pa u}_1 \norm{G^3}_{2n-1/2} \ls \ks \sqrt{\sd_{n}^0} \sqrt{\se_{n}^0 \sd_{n}^\sigma + \k \f_{n}}.
\end{equation}
The $G^4$ term is more delicate and must be handled with Proposition \ref{spatial_int_est}, which shows that 
\begin{equation}\label{sd_5}
\abs{\int_\Sigma ( \pa \eta - \sigma \Delta \pa \eta) \pa G^4 } \ls \ks \left[ \sqrt{\se_{n}^0 \sd_{n}^0 \sd_{n}^\sigma}  +   \sqrt{  \sd_{n}^\sigma \k \f_{n}} \right].
\end{equation}

We now combine  \eqref{sd_2}--\eqref{sd_5} with \eqref{sd_1} and \eqref{sd_1_2} to deduce that for $\abs{\alpha}=2n$ and $\alpha_0 =0$ we have that 
\begin{equation}\label{sd_7}
 \frac{d}{dt} \seb_\alpha^\sigma + \sdb_\alpha \ls \gamma \sd_{n}^0 + \ks \left[ \sqrt{\se_{n}^0 \sd_{n}^0 \sd_{n}^\sigma}  +   \sqrt{  \sd_{n}^\sigma \k \f_{n}} \right],
\end{equation}
which easily yields \eqref{sd_0} in this case.

\emph{Case 2 -- Everything else}  

We now consider the remaining cases, i.e. either $\abs{\alpha} \le 2n-1$ or else $\abs{\alpha} = 2n$ and $1\le \alpha_0 < n$.  In this case the $G^1,$ $G^2,$ and $G^3$ terms may be handled with  Theorem \ref{nlin_G} as in the Case 1 analysis above.  This  provides us with the bound
\begin{equation}\label{sd_8}
\abs{ \int_\Omega  \pa u \cdot \pa G^1 } +   \abs{\int_\Omega    \pa p \pa G^2} +  \abs{\int_\Sigma \pa G^3 \cdot \pa u} \ls \ks  \sqrt{\sd_{n}^\sigma} \sqrt{\se_{n}^0 \sd_{n}^\sigma + \k \f_{n}}.
\end{equation}
In this case the $G^4$ term does not require Proposition \ref{spatial_int_est} and may be estimated directly with  \eqref{n_G_00}:
\begin{equation}\label{sd_9}
\abs{\int_\Sigma ( \pa \eta - \sigma \Delta \pa \eta) \pa G^4 } \ls  \ks  \sqrt{\sd_{n}^\sigma} \sqrt{\se_{n}^0 \sd_{n}^\sigma + \k \f_{n}}.
\end{equation}
We may then combine \eqref{sd_8} and \eqref{sd_9} with \eqref{sd_1} and \eqref{sd_1_2} to deduce that when  either $\abs{\alpha} \le 2n-1$ or else $\abs{\alpha} = 2n$ and $1\le \alpha_0 < n$ we have the estimate \eqref{sd_0}.

\end{proof}

By combining Theorems \ref{time_derivative_est} and \ref{space_derivative_est} we get the following synthesized result.

\begin{thm}\label{energy_dissipation_est}
Assume that $\se_{n}^\sigma \le \delta$ for the universal $\delta \in (0,1)$ given by Proposition \ref{infty_ests}.
  We have the estimate
\begin{equation}\label{ed_0}
\frac{d}{dt} \left( \seb_n^\sigma  + \mathcal{H}_n \right) + \sdb_n  \ls \gamma \sd_{n}^0 +
\ks \left[ \sqrt{\se_{n}^0 } \sd_{n}^\sigma  +   \sqrt{  \sd_{n}^\sigma \k \f_{n}} \right]
\end{equation}
where 
\begin{equation}
 \mathcal{H}_n = \int_\Omega - \dt^{n-1} p  F^{2,n} J +  \hal \abs{\dt^n u}^2(J-1).
\end{equation}
\end{thm}

\subsection{Comparison estimates  }

Our goal now is to show that the full energy and dissipation, $\se_{n}^\sigma$ and $\sd_{n}^\sigma$, can be controlled by their horizontal counterparts, $\seb_{n}^\sigma$ and $\sdb_{n}$, up to some error terms that can be made small.  We begin with the result for the dissipation.

\begin{thm}\label{dissipation_comparision}
Assume that $\se_{n}^\sigma \le \delta$ for the universal $\delta \in (0,1)$ given by Proposition \ref{infty_ests}.
  Write 
\begin{multline}
 \y_n = \sum_{j=0}^{n-1}  \ns{ \dt^j  G^1}_{2n-2j-1}  +  \ns{ \dt^j G^2}_{2n -2j} +
 \ns{\dt^j  G^3}_{2n-2j -1/2}  \\
+  \ns{ G^4}_{2n-1/2} + \sigma^2 \ns{G^4}_{2n+1/2}  + \ns{\dt G^4}_{2n-2} + \sigma^2 \ns{\dt G^4}_{2n-3/2} + \sum_{j=2}^{n} \ns{\dt^j  G^4}_{2n-2j+1/2}
\end{multline}
and recall that $\sd_{n}^\sigma$ is defined by \eqref{D_def}.  Then we have the estimate 
\begin{equation}\label{dc_0}
 \sd_{n}^\sigma \ls \ks(\y_n + \sdb_{n}).
\end{equation}
\end{thm}

\begin{proof}
 We divide the proof into several steps.

\emph{Step 1 - Application of Korn's inequality}

Korn's inequality tells us that 
\begin{equation}\label{dc_1}
 \sum_{\substack{\alpha \in \mathbb{N}^{1+2} \\\abs{\alpha} \le 2n }} \ns{ \p^\alpha u}_1 \ls \sdb_{n}.
\end{equation}
Since $\p_1$ and $\p_2$ account for all spatial differential operators on $\Sigma$, we deduce from standard trace estimates that 
\begin{equation}\label{dc_2}
 \sum_{j=0}^n \ns{\dt^j u}_{H^{2n-2j+1/2}(\Sigma)} \ls  \sum_{\substack{\alpha \in \mathbb{N}^{1+2} \\\abs{\alpha} \le 2n }} \ns{ \p^\alpha u}_{H^{1/2}(\Sigma)} \ls \sdb_{n}.
\end{equation}

\emph{Step 2 - Elliptic estimates for the Stokes problem}

With \eqref{dc_2} in hand we can now use the elliptic theory associated to the Stokes problem to gain control of the velocity field and the pressure.  For $j=0,\dotsc,n-1$ we have that $(\dt^j u, \dt^j p, \dt^j \eta)$ solve the PDE 
\begin{equation}
\begin{cases}
 s \p_1 \dt^j u  + \diverge S (\dt^j p,\dt^j u) = \dt^j G^1 -  \dt^{j+1} u - s' \dt^j u_3 e_1  & \text{in }\Omega \\
 \diverge \dt^j u = \dt^j G^2 & \text{in } \Omega \\
 \dt^j u = \dt^j u \vert_{\Sigma}  & \text{on } \Sigma \\
 \dt^j u = 0 & \text{on } \Sigma_b.
\end{cases}
\end{equation}
We may then apply the elliptic estimates of Theorem \ref{stokes_dirichlet_est} to bound
\begin{multline}
 \ns{\dt^{n-1} u}_{3} + \ns{\nab \dt^{n-1} p}_1 \ls \ns{\dt^{n-1} G^1 }_1 +\ns{  \dt^{n} u }_1 + \ns{s' \dt^{n-1} u_3 e_1}_1 + \ns{\dt^{n-1} G^2}_2 + \ns{\dt^{n-1} u}_{H^{5/2}(\Sigma)} \\
 \ls \y_n + \sdb_n.
\end{multline}
The control of $\dt^{n-1} u$ provided by this bound then allows us to control $\dt^{n-2} u$ in a similar manner.  We thus proceed iteratively with Theorem \ref{stokes_dir_prob},  counting down from $n-1$ temporal derivatives to $0$ temporal derivatives, in order to deduce the estimate 
\begin{equation}\label{dc_3}
 \sum_{j=0}^{n-1} \ns{\dt^j u}_{2n-2j+1} + \ns{\nab \dt^j p}_{2n-2j-1} \ls \y_n + \sdb_{n}.
\end{equation}

\emph{Step 3 - Free surface function estimates}

Next we derive estimates for the free surface function.  We begin by isolating the dynamic boundary condition on $\Sigma$ to write 
\begin{equation}\label{dc_4}
\eta - \sigma \Delta \eta = S(p,u) e_3 \cdot e_3 -  G^3 \cdot e_3 = p - 2 \p_3 u_3   - G^3 \cdot e_3. 
\end{equation}
For $i = 1,2$ and $j=0,\dotsc,n-1$ we apply $\p_i \dt^j$ to see that 
\begin{equation}
\p_i \dt^j \eta - \sigma \Delta \p_i \dt^j \eta = \p_i \dt^j p - 2 \p_3 \p_i \dt^j u_3   - \p_i \dt^j G^3 \cdot e_3. 
\end{equation}
We then use this in Theorem \ref{cap_elliptic} and employ \eqref{dc_3} to see that 
\begin{multline}\label{dc_5}
  \ns{\p_i \dt^j \eta}_{2n-2j-3/2} + \sigma^2 \ns{\p_i \dt^j \eta}_{2n-2j+1/2} \ls \ns{\p_i \dt^j p - 2 \p_3 \p_i \dt^j u_3   - \p_i \dt^j G^3 \cdot e_3}_{H^{2n-2j-3/2}(\Sigma)} \\
\ls \ns{\nab \dt^j p}_{2n-2j-1} + \ns{\dt^ju}_{2n-2j+1} + \ns{\dt^j G^3}_{2n-2j-1/2} \ls \y_n + \sdb_{n}.
\end{multline}
We know that $\dt^j \eta$ has zero average over $\Sigma$ due to \eqref{avg_prop}, and so the Poincar\'{e} inequality on $\Sigma$ and \eqref{dc_5} then imply that
\begin{equation}\label{dc_6}
\sum_{j=0}^{n-1} \ns{\dt^j \eta}_{2n-2j-1/2} + \sigma^2 \ns{\dt^j \eta}_{2n-2j+3/2} \ls \sum_{j=0}^{n-1}\sum_{i=1}^2  \ns{\p_i\dt^j \eta}_{2n-3/2} + \sigma^2 \ns{\p_i \dt^j \eta}_{2n+1/2} \ls \y_n + \sdb_{n}.
\end{equation}

Next we estimate $\dt^j \eta$ for $j=1,\dotsc,n+1$ by employing the kinematic boundary condition
\begin{equation}
 \dt^{j+1} \eta = \dt^j u_3 -(\gamma/2) \p_1 \dt^j \eta + \dt^j G^4.
\end{equation}
We first use this and \eqref{dc_6} to bound
\begin{multline}\label{dc_7}
 \ns{\dt \eta}_{2n-1} \ls \ns{ u_3}_{H^{2n-1}(\Sigma)} + \gamma^2 \ns{\eta}_{2n} + \ns{ G^4}_{2n-1} \ls \ns{ u}_{2n-1/2} + \gamma^2 \ns{\eta}_{2n} + \y_n \\
 \ls \ks( \y_n + \sdb_{n})
\end{multline}
and then we multiply by $\sigma^2$ in order to derive the similar estimate
\begin{multline}\label{dc_8}
 \sigma^2 \ns{\dt \eta}_{2n+1/2} \ls \sigma^2 \ns{ u_3}_{H^{2n+1/2}(\Sigma)} + \gamma^2 \sigma^2 \ns{\eta}_{2n+3/2} + \sigma^2 \ns{ G^4}_{2n+1/2} \ls \sigma^2 \ns{u}_{2n+1} + \gamma^2 \sigma^2 \ns{\eta}_{2n+3/2} + \y_n \\
 \ls \ks( \y_n + \sdb_{n}). 
\end{multline}
Next we use a similar argument to control $\dt^2 \eta$:
\begin{multline}\label{dc_7_2}
 \ns{\dt^2 \eta}_{2n-2} \ls \ns{\dt u_3}_{H^{2n-2}(\Sigma)} + \gamma^2 \ns{\dt \eta}_{2n-1} + \ns{\dt G^4}_{2n-2} \ls \ns{\dt u}_{2n-2+1/2} + \gamma^2 \ns{\dt \eta}_{2n-1} + \y_n \\
 \ls \ks( \y_n + \sdb_{n})
\end{multline}
and
\begin{multline}\label{dc_8_2}
 \sigma^2 \ns{\dt^2 \eta}_{2n-3/2} \ls \sigma^2 \ns{\dt u_3}_{H^{2n-3/2}(\Sigma)} + \gamma^2 \sigma^2 \ns{\dt \eta}_{2n-1/2} + \sigma^2 \ns{\dt G^4}_{2n-3/2} \\
 \ls \sigma^2 \ns{\dt u}_{2n-1} + \gamma^2 \sigma^2 \ns{\dt \eta}_{2n-1/2} + \sigma^2 \y_n 
 \ls \ks( \y_n + \sdb_{n}). 
\end{multline}
With control of $\dt^2 \eta$ in hand we can iterate to obtain control of $\dt^j \eta$ for $j=3,\dotsc,n+1$.  This yields the estimate 
\begin{equation}\label{dc_9}
 \sum_{j=3}^{n+1} \ns{\dt^j \eta}_{2n-2j+5/2} \ls \ks(\y_n + \sdb_{n}).
\end{equation}

The free surface function terms remaining to control in $\sd_{n}^\sigma$ can be handled by considering the dot product of the dynamic boundary condition with $e_1$.  This reads
\begin{equation}
 \gamma \eta = [G^3 - S(p,u) e_3  ]\cdot e_1 = G^3 \cdot e_1 + \p_1 u_3 + \p_3 u_1.
\end{equation}
Applying $\dt^j$ for $j=0,\dotsc,n$ and using \eqref{dc_3} then provides us with the estimate 
\begin{equation}\label{dc_10}
 \gamma^2 \sum_{j=0}^{n-1} \ns{\dt^j \eta}_{2n-2j-1/2} \ls \sum_{j=0}^{n-1} \ns{G^3}_{2n-2j-1/2} + \ns{\dt^j u}_{H^{2n-2j+1/2}(\Sigma)} \ls \y_n + \sum_{j=0}^{n-1} \ns{\dt^j u}_{2n-2j+1} 
\ls \y_n + \sdb_{n}.
\end{equation}

Summing \eqref{dc_6}, \eqref{dc_7}--\eqref{dc_9}, and \eqref{dc_10} then shows that 
\begin{multline}\label{dc_11}
 \sum_{j=0}^{n-1}\left((1+\gamma^2) \ns{\dt^j \eta}_{2n-2j-1/2} + \sigma^2 \ns{\dt^j \eta}_{2n-2j+3/2} \right)  + \ns{\dt \eta}_{2n-1} + \sigma^2 \ns{\dt \eta}_{2n+1/2} 
\\ + \ns{\dt^2 \eta}_{2n-2} + \sigma^2 \ns{\dt^2 \eta}_{2n-3/2} + \sum_{j=3}^{n+1} \ns{\dt^j \eta}_{2n-2j+5/2} 
\ls \ks(\y_n + \sdb_{n}).
\end{multline}

\emph{Step 4 -- Improved pressure estimates}

We now return to \eqref{dc_4} with \eqref{dc_11} in hand in order to improve our estimates for the pressure.  Applying $\dt^j$ for $j=0,\dotsc,n-1$ shows that 
\begin{equation}
 \dt^j p = \dt^j \eta - \sigma \Delta \dt^j \eta + 2 \p_3 \dt^j u_3 + \dt^j G^3 \cdot e_3.
\end{equation}
We then use this with \eqref{dc_6} to bound
\begin{equation}
 \sum_{j=0}^{n-1} \ns{\dt^j p}_{H^0(\Sigma)} \ls \sum_{j=0}^{n-1} \ns{\dt^j \eta}_0 +\sigma^2 \ns{\dt^j \eta}_2 + \ns{\dt^j u}_2 + \ns{\dt^j G_3}_0 \ls \y_n + \sdb_{n}.
\end{equation}
We then apply a Poincar\'{e}-type inequality to see that
\begin{equation}
 \sum_{j=0}^{n-1} \ns{\dt^j p}_0 \ls  \sum_{j=0}^{n-1} \ns{\nab \dt^j p}_0 + \ns{\dt^j p}_{H^0(\Sigma)} \ls \y_n + \sdb_{n}.
\end{equation}
Hence 
\begin{equation}\label{dc_12}
 \sum_{j=0}^{n-1} \ns{\dt^j p}_{2n-2j} \ls  \sum_{j=0}^{n-1} \ns{\dt^j p}_{0} + \ns{\nab \dt^j p}_{2n-2j-1} \ls \y_n + \sdb_{n}.
\end{equation}

\emph{Step 5 -- Conclusion}
 
The estimate \eqref{dc_0} now follows by combining \eqref{dc_1}, \eqref{dc_3}, \eqref{dc_11}, and \eqref{dc_12}.

\end{proof}

Our next result proves the comparison result for the energy.

\begin{thm}\label{energy_comparision}
Assume that $\se_{n}^\sigma \le \delta$ for the universal $\delta \in (0,1)$ given by Proposition \ref{infty_ests}.
  Write 
\begin{equation}
\w_n = \sum_{j=0}^{n-1} \ns{ \dt^j  G^1}_{2n-2j-2}  +  \ns{ \dt^j G^2}_{2n -2j - 1} +
 \ns{\dt^j  G^3}_{2n-2j -3/2} + \ns{\dt^j G^4}_{2n-2j-1/2 }
\end{equation}
and recall that $\se_{n}^\sigma$ is defined by \eqref{E_def}.  Then we have the estimate 
\begin{equation}\label{ec_0}
 \se_{n}^\sigma \ls \ks(\w_n + \seb_{n}^\sigma).
\end{equation}
\end{thm}

\begin{proof}
We divide the proof into several steps.

\emph{Step 1 -- Initial free surface terms}

To begin we note that 
\begin{equation}
 \sum_{\substack{\alpha \in \mathbb{N}^{1+2} \\\abs{\alpha} \le 2n }}  \ns{\p^\alpha \eta}_{0} + \sigma \ns{ \nab \p^\alpha \eta }_{0} = \sum_{j=0}^n \ns{\dt^j \eta}_{2n-2j} + \sigma \ns{\nab \dt^j \eta}_{2n-2j}.
\end{equation}
Since $\dt^j \eta$ has vanishing average for each $j=0,\dotsc,n$ (due to  \eqref{avg_prop}) we can then use the Poincar\'{e} inequality to conclude that 
\begin{equation}\label{ec_1}
\sum_{j=0}^n \ns{\dt^j \eta}_{2n-2j} + \sigma \ns{\dt^j \eta}_{2n-2j+1} \ls  \sum_{j=0}^n \ns{\dt^j \eta}_{2n-2j} + \sigma \ns{\nab \dt^j \eta}_{2n-2j} \ls \seb_{n}^\sigma.
\end{equation}

\emph{Step 2 -- Elliptic estimates}

Next we write the PDE satisfied by $(\dt^j u, \dt^j p, \dt^j \eta)$ for $j=0,\dotsc,n-1$ as 
\begin{equation}
\begin{cases}
s \p_1 \dt^j u  + \diverge S(\dt^j p,\dt^j u)  = \dt^j G^1 - \dt^{j+1}u  & \text{in }\Omega \\
\diverge\dt^j  u = \dt^j G^2 &\text{in }\Omega \\
S(\dt^j p,\dt^j u) e_3  = (\dt^j  \eta - \sigma \Delta \dt^j  \eta)e_3 - \gamma \dt^j  \eta e_1 + \dt^j  G^3  & \text{on }  \Sigma \\
\dt^j  u =0 & \text{on } \Sigma_b.
\end{cases}
\end{equation}
This allows us to apply the elliptic estimate for the Stokes problem with stress boundary conditions, Theorem \ref{stokes_stress_est}, to bound (using \eqref{ec_1} along the way)
\begin{multline}
 \ns{\dt^{n-1} u}_2 + \ns{\dt^{n-1} p}_1 \ls \ns{\dt^{n-1} G^1}_0 + \ns{\dt^{n}u}_0 + \ns{\dt^{n-1} G^2}_1 + \ns{\dt^{n-1} \eta}_{1/2} + \sigma^2 \ns{\dt^{n-1} \eta}_{5/2} + \ns{\dt^{n-1} G^3}_{1/2} \\
\ls \w_n + \ks \seb_{n}^\sigma.
\end{multline}
We then proceed iteratively to estimate $(\dt^j u,\dt^j p)$ for $j=n-2,\dotsc,0$, employing \eqref{ec_1} and Theorem \ref{stokes_stress_est}.  This yields the bound
\begin{equation}\label{ec_2}
 \sum_{j=0}^{n-1} \ns{\dt^j u}_{2n-2j} + \ns{\dt^j p}_{2n-2j-1} \ls \w_n + \ks \seb_{n}^\sigma.
\end{equation}

\emph{Step 3 -- Improved estimates for time derivatives of the free surface function}

With the estimates \eqref{ec_2} in hand we can improve the estimates for the time derivatives of the free surface function by employing the kinematic boundary condition 
\begin{equation}
 \dt^{j+1} \eta = \dt^j u_3 -(\gamma/2) \p_1 \dt^j \eta + \dt^j G^4 \text{ for }j=0,\dotsc,n-1.
\end{equation}
This, trace theory, and \eqref{ec_1} and \eqref{ec_2} provide us with the estimates
\begin{equation}
  \ns{\dt \eta}_{2n-1}  \ls  \ns{ u }_{2n} + \gamma^2 \ns{ \eta}_{2n} + \ns{G^4}_{2n-1} \ls \w_n + \ks \seb_{n}^\sigma 
\end{equation}
and
\begin{equation}
 \sigma \ns{\dt \eta}_{2n-1/2}  \ls \sigma \ns{ u }_{2n} + \gamma^2 \sigma \ns{ \eta}_{2n+1/2} + \sigma \ns{G^4}_{2n-1/2} \ls  \ks(\w_n+ \seb_{n}^\sigma)  
\end{equation}
We then iterate this argument to control $\dt^j \eta$ for $j=2,\dotsc,n-1$.  This yields the bound
\begin{equation}
 \sum_{j=2}^{n} \ns{\dt^j \eta}_{2n-2j+3/2} \ls \ks(\w_n + \seb_{n}^\sigma). 
\end{equation}
Combining these estimates then shows that
\begin{equation}\label{ec_3}
 \ns{\dt \eta}_{2n-1} + \sigma \ns{\dt \eta}_{2n-1/2} +  \sum_{j=2}^{n} \ns{\dt^j \eta}_{2n-2j+3/2} \ls \ks(\w_n + \seb_{n}^\sigma).
\end{equation}

\emph{Step 4 -- Conclusion}

The estimate \eqref{ec_0} now follows by summing \eqref{ec_1}, \eqref{ec_2}, and \eqref{ec_3}.

\end{proof}

\section{The vanishing surface tension problem }\label{sec_nost_gwp}

In this section we complete the development of the a priori estimates for the vanishing surface tension problem and for the problem with zero surface tension.  With these estimates in hand we then prove Theorems \ref{gwp_vanish} and \ref{vanishing_st}, which establish the existence of global-in-time decaying solutions  and study the limit as surface tension vanishes.

\subsection{Preliminaries }

Here we record a simple preliminary estimate that will be quite useful in the subsequent analysis.

\begin{prop}\label{kf_ests}
For $N \ge 3$ we have that
\begin{equation}
 \k \ls \min\{\se_{N+2}^0, \sd_{N+2}^0\} \text{ and } \f_{N+2} \ls \se_{2N}^0.
\end{equation}
\end{prop}
\begin{proof}
The Sobolev embeddings and trace theory show that $\k \ls \ns{u}_{7/2} + \ns{\eta}_{5/2} \le \ns{u}_4 + \ns{\eta}_4$ and hence $\k \ls \se_{2}^0 \le \se_{N+2}^0$ and $\k \ls \sd_{2}^0 \le \sd_{N+2}^0.$  On the other hand, $\f_{N+2} = \ns{\eta}_{2N +5}$ and $2N + 5 \le 4N$ for integers $N \ge 3$, so $\f_{N+2} \le \se_{2N}^0$.
\end{proof}

\subsection{Transport estimate }

We now turn to the issue of establishing structured estimates for the highest derivatives of $\eta$ by appealing to the kinematic transport equation.  We begin by recording a general estimate for fractional derivatives of solutions to the transport equation, proved by Danchin \cite{danchin}.  Note that the result in \cite{danchin} is stated for $\Sigma = \R^2$, but it can be readily extended to periodic $\Sigma$ of the form we use: see for instance \cite{danchin_notes}.

\begin{lem}[Proposition 2.1 of \cite{danchin}]\label{sobolev_transport}
Let $\zeta$ be a solution to 
\begin{equation}\label{i_transport_eqn}
\begin{cases}
 \dt \zeta + w \cdot D \zeta = g & \text{in } \Sigma \times (0,T)  \\
 \zeta(t=0) = \zeta_0.
\end{cases}
\end{equation} 
Then there is a universal constant $C>0$ so that for any $0 \le s <2$
\begin{equation}
 \sup_{0\le r \le t} \norm{\zeta(r)}_{s} \le \exp\left(C \int_0^t \norm{D w(r)}_{3/2} dr \right) \left( \norm{\zeta_0}_{s} + \int_0^t \norm{g(r)}_{s}dr  \right).
\end{equation}
\end{lem}
\begin{proof}
Use $p=p_2 = 2$, $N=2$, and $\sigma=s$ in Proposition 2.1 of \cite{danchin} along with the embedding    $H^{3/2}  \hookrightarrow B^1_{2,\infty}\cap L^\infty.$  
\end{proof}

We now parlay Lemma \ref{sobolev_transport} into the desired estimate for the highest spatial derivatives of $\eta$.

\begin{thm}\label{specialized_transport_estimate}
Assume that $\se_{n}^\sigma \le \delta$ for the universal $\delta \in (0,1)$ given by Proposition \ref{infty_ests}.  Then 
\begin{multline}\label{ste_0}
\sup_{0\le r \le t}  \f_{2N}(r) \ls \exp\left(C \int_0^t \left(1 + \gamma \sqrt{\k(r)}\right) \sqrt{\k(r)} dr \right) \\
\times \left[ \f_{2N}(0)   +   t  \int_0^t (1+ \gamma \se_{2N}^0(r)) \sd_{2N}^0(r)dr  +  \left(\int_0^t (1+ \gamma \se_{2N}^0(r)) \sqrt{\k(r) \f_{2N}(r)}   dr\right)^2 \right].
\end{multline}
\end{thm}
\begin{proof}
We begin by introducing some notation.  Throughout the proof we write $u = \tilde{u} + u_3 e_3$ for $\tilde{u} = u_1 e_1 + u_2 e_3$, and we write $w = \tilde{u} + s(\eta) e_1.$ We write $D$ for the $2D$ gradient operator on $\Sigma$   Then $\eta$ solves the transport equation 
\begin{equation}\label{ste_99}
\dt \eta + w \cdot D \eta = u_3 \text{ on }\Sigma. 
\end{equation}
We may then use Lemma \ref{sobolev_transport} with $s = 1/2$ to estimate
\begin{equation}\label{ste_9}
 \sup_{0 \le r \le t} \norm{\eta(r)}_{1/2} \le \exp \left(C \int_0^t \norm{D w(r)}_{H^{3/2}(\Sigma)} dr \right) \left[ \norm{\eta_0}_{1/2} + \int_0^t \norm{u_3(r)}_{H^{1/2}(\Sigma)}dr  \right].
\end{equation}
We estimate the term in the exponential by using the fact that $H^{3/2}(\Sigma)$ is an algebra:
\begin{equation}
 \norm{Dw}_{3/2} \ls \norm{D u}_{H^{3/2}(\Sigma)} + \gamma \norm{\eta D \eta}_{3/2} \ls \norm{D u}_{H^{3/2}(\Sigma)} + \gamma \norm{\eta}_{3/2} \norm{\eta}_{5/2} \ls (1 + \gamma \sqrt{\k }) \sqrt{\k},
\end{equation}
where $\k$ is as defined in \eqref{K_def}.  We may also use trace theory to bound  $\norm{u_3(r)}_{H^{1/2}(\Sigma)} \ls \sd_{2N}(r).$  This allows us to square both sides of \eqref{ste_9} and utilize Cauchy-Schwarz to deduce that
\begin{equation}\label{ste_1}
 \sup_{0 \le r \le t} \ns{\eta(r)}_{1/2} \ls \exp \left(2C \int_0^t (1 + \gamma \sqrt{\k(r)}) \sqrt{K(r)} dr \right) \left[ \ns{\eta_0}_{1/2} + t \int_0^t \sd_{2N}^0(r) dr  \right]. 
\end{equation}

Next we derive a higher regularity version of the estimate \eqref{ste_1}.  To this end we choose any multi-index  $\alpha \in \mathbb{N}^{2}$ with $\abs{\alpha} = 4N$, and we apply the operator $\pa$ to \eqref{ste_99} to see that $\pa \eta$ solves the transport equation
\begin{equation} 
 \dt (\partial^\alpha  \eta) + w \cdot D (\partial^\alpha  \eta) = \partial^\alpha u_3  -
\sum_{0 < \beta \le \alpha } C_{\alpha,\beta} \partial^\beta  w \cdot D \partial^{\alpha-\beta}  \eta
:=G^\alpha
\end{equation}
with the initial condition $\pa \eta_0$.  We then again apply Lemma \ref{sobolev_transport} with $s=1/2$ to find that
\begin{equation}\label{ste_2}
 \sup_{0\le r \le t} \norm{\partial^\alpha \eta(r)}_{1/2} \le \exp\left(C \int_0^t \norm{D w(r)}_{H^{3/2}(\Sigma)} dr \right) \left[ \norm{\pa \eta_0}_{1/2} + \int_0^t \norm{G^\alpha (r)}_{1/2}dr  \right].
\end{equation}
We will now estimate $\norm{G^\alpha }_{1/2}$.  In doing so we will need to write
\begin{equation}\label{ste_10}
 \p^\beta s(\eta) = -\frac{\gamma}{2} \sum_{0 \le \delta \le \beta} C_{\delta,\beta} \p^\delta \eta \p^{\beta - \delta} \eta.
\end{equation}

For $\beta \in \mathbb{N}^2$ satisfying $2N+1 \le \abs{\beta} \le 4N$ we may apply Lemma \ref{sobolev_product_1} with $s_1 = r = 1/2$ and $s_2=2$ to bound
\begin{equation}
\norm{\p^\beta w D \p^{\alpha-\beta} \eta}_{1/2} \ls   \norm{\p^\beta w}_{1/2} \norm{D \p^{\alpha-\beta} \eta}_{2} \ls  \left(\norm{\p^\beta u}_{H^{1/2}(\Sigma)} + \norm{\p^\beta s(\eta)}_{1/2}  \right) \norm{D \p^{\alpha-\beta} \eta}_{2}.
\end{equation}
To handle the $s(\eta)$ term we again employ Lemma \ref{sobolev_product_1} with $s_1 = r=1/2$ and $s_2 =2$ along with the expansion \eqref{ste_10} to estimate 
\begin{equation}
  \norm{\p^\beta s(\eta)}_{1/2} \ls \gamma \sum_{\ell =0 }^{\lfloor \abs{\beta}/2 \rfloor} \norm{\eta}_{2+\ell}  \norm{\eta}_{\abs{\beta}-\ell +1/2} \ls \gamma \norm{\eta}_{\lfloor \abs{\beta}/2 \rfloor +2} \norm{\eta}_{\abs{\beta}-1/2} + \gamma \norm{\eta}_2  \norm{\eta}_{\abs{\beta}+1/2}
\end{equation}
This and trace theory then imply that
\begin{multline}\label{ste_3}
\sum_{\substack{0 < \beta \le \alpha \\ 2N+1 \le \abs{\beta} \le 4N }}  \norm{ C_{\alpha,\beta} \partial^\beta  w \cdot D \partial^{\alpha-\beta}  \eta}_{1/2} \ls \left( \norm{u}_{4N+1} + \gamma \norm{\eta}_{2N+2} \norm{\eta}_{4N-1/2} + \gamma \norm{\eta}_2 \norm{\eta}_{4N+1/2}  \right)  \norm{\eta}_{2N+3} \\
\ls \left[ \left(1 + \gamma \sqrt{\se_{2N}^0}  \right) \sqrt{\sd_{2N}^0}  + \gamma \sqrt{\k \f_{2N}} \right]    \sqrt{\se_{2N}^0}.
\end{multline}
On the other hand, if $\beta$ satisfies $1 \le \abs{\beta} \le 2N$ then we use Lemma \ref{sobolev_product_1} to bound
\begin{equation}
\norm{\p^\beta w D \p^{\alpha-\beta} \eta}_{1/2} \ls \left(  \norm{\p^\beta \tilde{u}}_{H^{2}(\Sigma)} + \gamma \norm{\p^\beta s(\eta)}_{2} \right)\norm{D \p^{\alpha-\beta} \eta}_{1/2}.
\end{equation}
Since $H^2(\Sigma)$ is an algebra we can then use \eqref{ste_10} bound
\begin{equation}
 \norm{\p^\beta s(\eta)}_2 \ls \gamma \ns{\eta}_{\abs{\beta}+2}.
\end{equation}
From these we deduce that
\begin{multline}\label{ste_4}
\sum_{\substack{0 < \beta \le \alpha \\ 1 \le \abs{\beta} \le 2N }}  \norm{ C_{\alpha,\beta} \partial^\beta  w \cdot D \partial^{\alpha-\beta}  \eta}_{1/2} \ls \left( \norm{u}_{2N+3} + \gamma \ns{\eta}_{2N+2} \right) \norm{\eta}_{4N-1/2} \\
+ \left(\norm{D \tilde{u}}_{H^{2}(\Sigma)}  + \gamma \ns{\eta }_3 \right) \norm{\eta}_{4N+1/2} 
\ls \sqrt{\se_{2N}^0 \sd_{2N}^0} + (1+ \gamma \sqrt{\se_{2N}^0}) \sqrt{\k \f_{2N}}.
\end{multline}
The only remaining term in $G^\alpha$ is $\pa u_3$, which we estimate with trace theory:
\begin{equation}\label{ste_5}
 \norm{\pa u_3}_{H^{1/2}(\Sigma)} \ls \norm{ u_3}_{4N+1} \ls \sqrt{\sd_{2N}^0}.
\end{equation}
We may then combine \eqref{ste_3}, \eqref{ste_4}, and \eqref{ste_5} for
\begin{equation}\label{ste_6}
 \norm{G^\alpha}_{1/2} \ls \sqrt{\sd_{2N}} +  (1+ \gamma\sqrt{\se_{2N}}) \left( \sqrt{ \se_{2N}^0 \sd_{2N}^0} + \sqrt{\k \f_{2N}}\right).
\end{equation}

Returning now to \eqref{ste_2}, we square both sides and employ \eqref{ste_6} and our previous estimate of the term in the exponential to find that
\begin{multline}\label{ste_7}
 \sup_{0\le r \le t} \ns{\partial^\alpha \eta(r)}_{1/2} \le \exp\left(2 C \int_0^t (1 + \gamma \sqrt{\k(r)}) dr \right) \\
\times \left[ \ns{\pa \eta_0}_{1/2} + t \int_0^t (1+ \gamma \se_{2N}^0(r)) \sd_{2N}^0(r)dr  +  \left(\int_0^t (1+ \gamma \se_{2N}^0(r)) \sqrt{\k(r) \f_{2N}(r)}   dr\right)^2 \right]. 
\end{multline}
Then the estimate \eqref{ste_0} follows by summing \eqref{ste_7} over all $\abs{\alpha}=4N$, adding the resulting inequality to \eqref{ste_1}, and using the fact that $\ns{\eta}_{4N+1/2} \ls \ns{\eta}_{1/2} + \sum_{\abs{\alpha}=4N} \ns{\pa \eta}_{1/2}$. 
\end{proof}

Next we show that if we know a priori that $\g_{2N}$ is small, then in fact we can estimate $\f_{2N}$ in a stronger form than in Theorem \ref{sobolev_transport}.

\begin{thm}\label{f_bound}
Let $\g_{2N}$ be defined by \eqref{G_def} for $N \ge 3$.  There exists a universal $\delta \in (0,1)$ such that if $\g_{2N}(T) \le \delta$ and $\gamma \le 1$, then  
\begin{equation}\label{f_b_0}
\sup_{0 \le r \le t} \f_{2N}(r) \ls \f_{2N}(0) + t \int_0^t \sd_{2N}^0(r) dr 
\end{equation}
for all $0 \le t \le T$.
\end{thm}
\begin{proof}
According to Proposition \ref{kf_ests} and the assumed bounds, we may estimate 
\begin{equation}\label{f_b_1}
 \int_0^t  \left(1 + \gamma \sqrt{\k(r)}\right) \sqrt{\k(r)}  dr \ls \int_0^t \sqrt{\se_{N+2}^0(r)} dr \le  \sqrt{\delta} \int_0^\infty  \frac{1}{(1+r)^{2N-4}}dr  \ls \sqrt{\delta}.
\end{equation}
Since $\delta \in (0,1)$ this implies that for any universal $C>0$, 
\begin{equation}\label{f_b_2}
 \exp\left(C  \int_0^t  \left(1 + \gamma \sqrt{\k(r)}\right) \sqrt{\k(r)}  dr  \right) \ls 1.
\end{equation}
Similarly, 
\begin{equation}\label{f_b_3}
 \left(\int_0^t (1+ \gamma \se_{2N}(r)) \sqrt{\k(r) \f_{2N}(r)} dr \right)^2 \ls \left( \sup_{0\le r \le t}  \f(r) \right) \left(\int_0^t \sqrt{\se_{N+2}^0(r)}dr  \right)^2 
\ls \left( \sup_{0\le r \le t}  \f(r) \right) \delta.
\end{equation}
Then \eqref{f_b_1}--\eqref{f_b_3} and  Theorem \ref{specialized_transport_estimate} imply that
\begin{equation}
\sup_{0\le r \le t}  \f_{2N}(r) \le  
C \left( \f_{2N}(0) +   t \int_0^t \sd_{2N}^0(r) dr\right)   + C \delta \left( \sup_{0\le r \le t}  \f_{2N}(r) \right),
\end{equation}
for some $C>0$.  Then if $\delta$ is small enough so that $ C \delta \le 1/2$, we may absorb the right-hand $\f_{2N}$ term onto the left and deduce \eqref{f_b_0}.

\end{proof}

\subsection{A priori estimates for $\g_{2N}$ }

Our goal now is to complete our a priori estimates for $\g_{2N}$.  We start with the bounds of the high-tier terms and $\f_{2N}$.

\begin{thm}\label{g_est_top}
There exist $\delta_0,\gamma_0 \in (0,1)$ such that if $0 \le \gamma < \gamma_0$, $0\le \sigma \le 1$, and $\g_{2N}(T) \le \delta \le \delta_0$, then 
\begin{equation}\label{get_0}
\sup_{0 \le r \le t} \se_{2N}^\sigma(r) + \int_0^t \sd_{2N}^\sigma(r)dr + \sup_{0 \le r \le t} \frac{\f_{2N}(r)}{1+r} \ls \se_{2N}^\sigma(0) + \f_{2N}(0)
\end{equation}
for all $0 \le t \le T$.
\end{thm}
\begin{proof}

We first assume that $\delta_0$ is as small as in Proposition \ref{infty_ests}.  We also assume that $\delta_0$ is as small as in Proposition \ref{Hn_estimate} so that $\abs{\H_{2N}} \ls (\se_{2N}^0)^{3/2}$.   

Note that the assumed bounds on $\sigma$ and $\gamma$ allow us to estimate $\ks \ls 1$ in any inequality in which a term of the form $\ks$ appears.  We will employ this estimate throughout the proof in order to remove the appearances of constants $\ks$.

We invoke Theorems \ref{energy_comparision} and \ref{dissipation_comparision} with $n = 2N$ in order to bound 
\begin{equation}\label{get_1}
 \se_{2N}^\sigma \ls \w_{2N} + \seb_{2N}^\sigma \text{ and } 
 \sd_{2N}^\sigma \ls \y_{2N} + \sdb_{2N}.
\end{equation}
According to Theorem \ref{nlin_G} we may then bound
\begin{equation}\label{get_2}
 \w_{2N} \ls \se_{2N}^0 \se_{2N}^\sigma + \k \f_{2N} \text{ and } \y_{2N} \ls \se_{2N}^0 \sd_{2N}^\sigma + \k \f_{2N}.
\end{equation}
Upon combining \eqref{get_1} and \eqref{get_2} with the above bound for $\mathcal{H}_{2N}$, we find that
\begin{equation}
 \se_{2N}^\sigma \ls (\seb_{2N}^\sigma + \mathcal{H}_{2N}) + \se_{2N}^0 \se_{2N}^\sigma + (\se_{2N}^0)^{3/2} + \k \f_{2N} \text{ and } \sd_{2N}^\sigma \ls \sdb_{2N} + \se_{2N}^0 \sd_{2N}^\sigma + \k \f_{2N},
\end{equation}
and consequently, if $\delta$ is assumed to be small enough we may absorb the $\se_{2N}^0 \se_{2N}^\sigma + (\se_{2N}^0)^{3/2}$ and $\se_{2N}^0 \sd_{2N}^\sigma$ terms onto the left to arrive at the bounds
\begin{equation}\label{get_3}
 \se_{2N}^\sigma \ls \left( \seb_{2N}^\sigma + \mathcal{H}_{2N} \right) + \k \f_{2N} \text{ and } \sd_{2N}^\sigma \ls \sdb_{2N} + \k \f_{2N}.
\end{equation}

We apply Theorem \ref{energy_dissipation_est} with $n = 2N$ and integrate in time from $0$ to $t$ to see that
\begin{multline}
 \left(\seb_{2N}^\sigma(t)  + \mathcal{H}_{2N}(t) \right) + \int_0^t \sdb_{2N}(r) dr  \ls \left(\seb_{2N}^\sigma(0)  + \mathcal{H}_{2N}(0) \right)  + \int_0^t \gamma \sd_{2N}^0(r)dr \\
 +
\int_0^t \sqrt{\se_{2N}^0(r) } \sd_{2N}^\sigma(r) dr  +  \int_0^t \sqrt{  \sd_{2N}^\sigma(r) \k(r) \f_{2N}(r)} dr.
\end{multline}
We then combine this with the estimates in \eqref{get_3} to arrive at the refined bound
\begin{multline}\label{get_4}
 \se_{2N}^\sigma(t) + \int_0^t \sd_{2N}^\sigma(r) dr \ls \se_{2N}^\sigma(0) + \int_0^t \gamma \sd_{2N}^0(r) dr + \int_0^t \sqrt{\se_{2N}^0(r) } \sd_{2N}^\sigma(r) dr  \\
   +  \int_0^t \left(  \k(r) \f_{2N}(r)  +  \sqrt{  \sd_{2N}^\sigma(r) \k(r) \f_{2N}(r)} \right)dr.
\end{multline}

We now turn our attention to the $\k \f_{2N}$ terms appearing on the right side of \eqref{get_4}.  To handle these we first note that $\k \ls \se_{N+2}^0$, as is shown in Proposition \ref{kf_ests}.  Thus 
\begin{equation}\label{get_5}
 \k(r) \ls \se_{N+2}^0(r) =   \frac{1}{(1+r)^{4N-8}} (1+r)^{4N-8} \se_{N+2}^0 \le \frac{1}{(1+r)^{4N-8}} \g_{2N}(T) \le \frac{\delta}{(1+r)^{4N-8}}.
\end{equation}
Next we use Theorem \ref{f_bound} to see that for $0 \le r \le t$ we can estimate 
\begin{equation}\label{get_6}
 \f_{2N}(r) \ls \f_{2N}(0)  + (1+r)\int_0^r \sd_{2N}^0(s) ds.
\end{equation}
We may then combine \eqref{get_5} and \eqref{get_6} to estimate 
\begin{multline}\label{get_7}
 \int_0^t  \k(r) \f_{2N}(r)dr \ls \delta \int_0^t \left( \frac{\f_{2N}(0)}{(1+r)^{4N-8}}  + \frac{1}{(1+r)^{4N-7} }\int_0^r \sd_{2N}^0(s) ds \right) \\
 \ls \delta \f_{2N}(0)  \int_0^\infty \frac{dr}{(1+r)^{4N-8}} + \delta \left( \int_0^t \sd_{2N}^0(r) dr \right) \left( \int_0^\infty \frac{dr}{(1+r)^{4N-7}}\right) \\
 \ls \delta \f_{2N}(0) + \delta \int_0^T \sd_{2N}^0(r) dr,
\end{multline}
where here we have used the fact that $N \ge 3$ to guarantee that $(1+r)^{7-4N}$ and $(1+r)^{8-4N}$ are integrable on $(0,\infty)$.  Similarly, we may estimate 
\begin{multline}\label{get_8}
 \int_0^t  \sqrt{  \sd_{2N}^\sigma(r) \k(r) \f_{2N}(r)} dr \le \left( \int_0^t \sd_{2N}^\sigma(r) dr\right)^{1/2} \left(\int_0^t \k(r) \f_{2N}(r) dr \right)^{1/2}  \\
\ls \left( \int_0^t \sd_{2N}^\sigma(r) dr\right)^{1/2} \left( \delta \f_{2N}(0) + \delta \int_0^T \sd_{2N}^0(r) dr\right)^{1/2} \ls \sqrt{\delta} \f_{2N}(0) + \sqrt{\delta} \int_0^t \sd_{2N}^\sigma(r) dr,
\end{multline}
where in the last inequality we have used the bound $\sqrt{a + b} \le \sqrt{a} + \sqrt{b}$ for $a,b \ge 0$, as well as Cauchy's inequality.

Now we plug \eqref{get_7} and \eqref{get_8} into \eqref{get_4}, bound $\se_{2N}^0 \le \g_{2N}^\sigma \le \delta$, and use the fact that $\delta < 1$ implies $\delta \le \sqrt{\delta}$, to arrive at the bound
\begin{equation}
 \se_{2N}^\sigma(t) + \int_0^t \sd_{2N}^\sigma(r) dr \ls \se_{2N}^\sigma(0) + \f_{2N}(0) + \int_0^t (\gamma + \sqrt{\delta}) \sd_{2N}^\sigma(r) dr.
\end{equation}
Thus if $\gamma_0,\delta_0 \in (0,1)$ are chosen to be small enough, we may absorb the $(\gamma + \sqrt{\delta}) \sd_{2N}^\sigma(r)$ term onto the left to deduce that 
\begin{equation}\label{get_9}
 \se_{2N}^\sigma(t) + \int_0^t \sd_{2N}^\sigma(r) dr \ls \se_{2N}^\sigma(0) + \f_{2N}(0).
\end{equation}
Upon combining \eqref{get_6} and \eqref{get_9} we then easily conclude that \eqref{get_0} holds.

\end{proof}

Next we establish the algebraic decay results for the low-tier energy.

\begin{thm}\label{g_est_btm}
There exist $\delta_0,\gamma_0 \in (0,1)$ such that if $0 \le \gamma < \gamma_0$, $0\le \sigma \le 1$, and $\g_{2N}(T) \le \delta \le \delta_0$, then 
\begin{equation}\label{geb_0}
\sup_{0 \le r \le t} (1+r)^{4N-8}\se_{N+2}^\sigma(r)  \ls \se_{2N}^\sigma(0) + \f_{2N}(0)
\end{equation}
for all $0 \le t \le T$.
\end{thm}
\begin{proof}

Assume $\delta_0$ is as small as in Propositions \ref{infty_ests} and \ref{Hn_estimate}.  The latter allows us to estimate 
\begin{equation}\label{geb_1}
\abs{\H_{N+2}} \ls (\se_{N+2}^0)^{3/2} \ls \sqrt{\se_{2N}^0} \se_{N+2}^0 
\end{equation}
since $N \ge 3$.  As in the proof of Theorem \ref{g_est_top} we will use the bounds on $\sigma$ and $\gamma$ to bound $\ks \ls 1$ throughout the rest of the proof.

Theorems \ref{energy_comparision} and \ref{dissipation_comparision} with $n = N+2$, together with Theorem \ref{nlin_G} and Proposition \ref{kf_ests} provide the bounds 
\begin{equation}\label{geb_2}
\begin{split}
\se_{N+2}^\sigma & \ls  \left(\seb_{N+2}^\sigma + \mathcal{H}_{N+2}\right) + \se_{N+2}^0 \se_{N+2}^\sigma + \sqrt{\se_{2N}^0}\se_{N+2}^0 + \se_{2N}^0 \se_{N+2}^0\text{ and } \\
 \sd_{N+2}^\sigma &\ls \sdb_{N+2} + \se_{N+2}^0 \sd_{N+2}^\sigma + \se_{2N}^0 \se_{N+2}^0.
\end{split}
\end{equation}
Consequently, if $\delta$ is assumed to be small enough we may absorb the $\se_{N+2}^0 \se_{N+2}^\sigma + \sqrt{\se_{2N}^0}\se_{N+2}^0 + \se_{2N}^0 \se_{N+2}^0$ and $\se_{N+2}^0 \sd_{N+2}^\sigma+ \se_{2N}^0 \se_{N+2}^0$ terms onto the left to arrive at the bounds
\begin{equation}\label{geb_3}
 \se_{N+2}^\sigma  \ls  \left(\seb_{N+2}^\sigma + \mathcal{H}_{N+2}\right)  \ls \se_{N+2}^\sigma \text{ and } 
 \sd_{N+2}^\sigma \ls \sdb_{N+2} \le \sd_{N+2}^\sigma.
\end{equation}

Now set 
\begin{equation}\label{geb_4}
 \theta = \frac{4N-8}{4N-7} \in (0,1).
\end{equation}
We claim that we have the interpolation estimate 
\begin{equation}\label{geb_5}
 \se_{N+2}^\sigma \ls \left( \sd_{N+2}^\sigma \right)^\theta \left(\se_{2N}^\sigma \right)^{1-\theta}.
\end{equation}
For most of the terms appearing in $\se_{N+2}^\sigma$ this is a simple matter.  Indeed, the definitions of $\sd_{N+2}^\sigma$ and $\se_{2N}^\theta$ allow us to trivially estimate
\begin{multline}\label{geb_6}
\sum_{j=0}^{N+2} \ns{\dt^j u}_{2(N+2)-2j} + \sum_{j=0}^{N+1} \ns{\dt^j p}_{2(N+2)-2j-1}   + \ns{\dt \eta}_{2(N+2)-1}   +\sum_{j=2}^{N+2} \ns{\dt^j \eta}_{2(N+2)-2j+3/2} \\
\le \left( \sd_{N+2}^\sigma \right)^\theta \left( \se_{2N}^\sigma \right)^{1-\theta}.
\end{multline}
To handle these remaining terms we must use Sobolev interpolation.  We begin with the most important term, which actually dictates the choice of $\theta$:
\begin{equation}\label{geb_7}
\ns{\eta}_{2(N+2)}\le  \norm{\eta}_{2(N+2)-1/2} ^{2\theta} \norm{\eta}_{4N}^{2(1-\theta)}
\ls ( {\mathcal{D}_{N+2}^\sigma})^\theta({\mathcal{E}_{2N}^\sigma})^{1-\theta}.
\end{equation}
Next we bound
\begin{multline}\label{geb_8}
 \sigma \ns{\dt \eta}_{2(N+2)-1/2} \ls \sigma \norm{\dt \eta}_{2(N+2) +1/2} \norm{\dt \eta}_{2(N+2)-3/2} \\
 = \left(\sigma^2 \ns{\dt \eta}_{2(N+2) +1/2} \right)^{1/2} \left(\ns{\dt \eta}_{2(N+2)-3/2}\right)^{\frac{4N-9}{8N-14}} \left( \ns{\dt \eta}_{2(N+2)-3/2}\right)^{\frac{1}{4N-7}} \\
\ls  \left(\sigma^2 \ns{\dt \eta}_{2(N+2) +1/2} \right)^{1/2} \left(\ns{\dt \eta}_{2(N+2)-1}\right)^{\frac{4N-9}{8N-14}} \left( \ns{\dt \eta}_{4N}\right)^{\frac{1}{4N-7}} \\
\ls
\left(\sd_{N+2}^\sigma \right)^{1/2} \left( \sd_{N+2}^\sigma \right)^{\frac{4N-9}{8N-14}} \left(\se_{2N}^\sigma\right)^{\frac{1}{4N-7}} 
= \left(\sd_{N+2}^\sigma \right)^{\theta}  \left(\se_{2N}^\sigma\right)^{1-\theta}. 
\end{multline}
Finally, we use the fact that $\sigma \le 1$ to trivially bound 
\begin{equation}\label{geb_9}
\sigma \sum_{j=2}^{N+2} \ns{\dt^j \eta}_{2(N+2)-2j+1} \le   \sum_{j=2}^{N+2} \ns{\dt^j \eta}_{2(N+2)-2j+1}  \le  (\sd_{N+2}^\sigma)^\theta (\se_{2N}^\sigma)^{1-\theta}. 
\end{equation}
Thus, upon summing \eqref{geb_6}--\eqref{geb_9} we find that the estimate \eqref{geb_5} holds, which proves the claim.

Next we employ Theorem \ref{energy_dissipation_est} with $n=N+2$ in conjunction with Proposition \ref{kf_ests}  see that 
\begin{equation}
\frac{d}{dt} \left( \seb_{N+2}^\sigma  + \mathcal{H}_{N+2} \right) + \sdb_{N+2}  \ls \gamma \sd_{N+2}^0 +
 \sqrt{\se_{N+2}^0 } \sd_{N+2}^\sigma  +   \sqrt{\se_{2N}^0}  \sd_{N+2}^\sigma. 
\end{equation}
We use this together with the bound $\g_{2N}^\sigma(T) \le \delta$ and the dissipation estimates in \eqref{geb_3} to the estimate 
\begin{equation}
\frac{d}{dt} \left( \seb_{N+2}^\sigma  + \mathcal{H}_{N+2} \right) + \sdb_{N+2}  \ls (\gamma + \sqrt{\delta}) \sdb_{N+2}.  
\end{equation}
Consequently, if $\gamma_0$ and $\delta_0$ are taken to be sufficiently small, we may absorb the $\sdb_{N+2}$ term onto the left of this inequality.  Doing so and again invoking the dissipation bounds of \eqref{geb_3} then shows that 
\begin{equation}\label{geb_10}
\frac{d}{dt} \left( \seb_{N+2}^\sigma  + \mathcal{H}_{N+2} \right) + C_0 \sd_{N+2}^\sigma \le 0  
\end{equation}
for a universal constant $C_0 >0$.  We then use the energy estimate in \eqref{geb_3} to to rewrite \eqref{geb_5} as 
\begin{equation}
 \left(\seb_{N+2}^\sigma  + \mathcal{H}_{N+2}\right)^{1/\theta} \ls \sd_{N+2}^\sigma \left( \se_{2N}^\sigma \right)^{(1-\theta)/\theta)}. 
\end{equation}
We chain this together with the estimate of Theorem \ref{g_est_top} to write 
\begin{equation}\label{geb_11}
\frac{C_1}{\mathcal{M}_0^s} \left(\seb_{N+2}^\sigma  + \mathcal{H}_{N+2}\right)^{1+s} \le \sd_{N+2}^\sigma 
\end{equation}
for  $C_1 >0$ a universal constant, 
\begin{equation}
 \mathcal{M}_0 := \se_{2N}(0) + \f_{2N}(0)
\text{ and }
s := \frac{1}{4N-8}.
\end{equation}
Upon combining \eqref{geb_10} and \eqref{geb_11}, we arrive at the differential inequality 
\begin{equation}\label{geb_12}
 \frac{d}{dt} \left( \seb_{N+2}^\sigma  + \mathcal{H}_{N+2} \right) + \frac{C_0 C_1}{\mathcal{M}_0^s} \left( \seb_{N+2}^\sigma  + \mathcal{H}_{N+2} \right)^{1+s} \le 0.
\end{equation}

With the inequality \eqref{geb_12} in hand we may integrate and argue as in the proofs of Theorem 7.7 of \cite{guo_tice_per} or Proposition 8.4 of \cite{jang_tice_wang_gwp} to deduce that 
\begin{equation}\label{geb_13}
 \sup_{0\le r \le t} (1+r)^{4N-8} \left( \seb_{N+2}^\sigma(r)  + \mathcal{H}_{N+2}(r) \right) \ls \mathcal{M}_0 = \se_{2N}(0) + \f_{2N}(0).
\end{equation}
Then \eqref{geb_13} and the energy bound in \eqref{geb_3} yield \eqref{geb_0}.
\end{proof}

As the final step in our a priori estimates for $\g_{2N}^\sigma$ we synthesize Theorems \ref{g_est_top} and \ref{g_est_btm}.

\begin{thm}\label{apriori_vanish}
There exist $\delta_0,\gamma_0 \in (0,1)$ such that if $0 \le \gamma < \gamma_0$, $0\le \sigma \le 1$, and $\g_{2N}(T) \le \delta \le \delta_0$, then 
\begin{equation}
\g_{2N}^\sigma(T) \ls \se_{2N}^\sigma(0) + \f_{2N}(0).
\end{equation}
\end{thm}
\begin{proof}
We simply combine the estimates of Theorems \ref{g_est_top} and \ref{g_est_btm}.
\end{proof}

\subsection{Main results for the vanishing surface tension problem  }

Now that we have the a priori estimates of of Theorem \ref{apriori_vanish} in hand, we may prove Theorems \ref{gwp_vanish} and \ref{vanishing_st} by following previously developed arguments.  For the sake of brevity we will omit full details and simply refer to the existing arguments.

\begin{proof}[Proof of Theorem \ref{gwp_vanish}]
The stated results follow by combining the local well-posedness theory, Theorem \ref{lwp}, with the a priori estimates of Theorem \ref{apriori_vanish} and a continuation argument.  The details of the continuation argument may be fully developed by following the arguments elaborated in  Theorem 1.3 of \cite{guo_tice_per} or Theorem 2.3 of \cite{jang_tice_wang_gwp}. 
\end{proof}

\begin{proof}[Proof of Theorem \ref{vanishing_st}]

The results follow from the estimates of Theorem \ref{gwp_vanish} and standard compactness arguments.  See Theorem 1.2 of \cite{tan_wang} or Theorem 2.9 of \cite{jang_tice_wang_gwp} for details.
\end{proof}

\section{The fixed surface tension problem }\label{sec_st_gwp}

In this section we study the problem \eqref{ns_geometric} in the case of a fixed $\sigma >0$.  We develop a priori estimates and then present the proof of Theorem \ref{gwp_st}.

\subsection{A priori estimates for $\S_\lambda$ }

In order to prove Theorem \ref{gwp_st} we will introduce the following notation when $\lambda \in (0,\infty)$:
\begin{equation}\label{S_lambda_def}
 \S_\lambda(T) := \sup_{0 \le t \le T} e^{\lambda t} \se_{2}^\sigma(T) + \int_0^T \sd_{2}^\sigma(t) dt 
\end{equation}

Our next result is the analog of Proposition \ref{kf_ests} that we will use in the case of positive surface tension.

\begin{prop}\label{kf_ests_st}
For $n \ge 3$ we have that
\begin{equation}
 \k \ls \min\{\se_{2}^0, \sd_{2}^0\} \text{ and } \f_{2} \le \min\{\sigma^{-1} \se_{2}^\sigma,\sigma^{-2} \sd_{2}^\sigma \}.
\end{equation}
\end{prop}
\begin{proof}
The first inequality follows as in the proof of Proposition \ref{kf_ests}.  The second follows since $\f_{4} = \ns{\eta}_{4+1/2} \le \ns{\eta}_{5} \le \min\{\sigma^{-1} \se_{2}^\sigma, \sigma^{-1} \se_{2}^\sigma\}$
\end{proof}

We now develop the main a priori estimates with surface tension.

\begin{thm}\label{apriori_st}
There exist $\delta_0,\gamma_0 \in (0,1)$ such that if $0 \le \gamma < \gamma_0$ and  $\S_{0}(T) \le \delta \le \delta_0$, then there exists $\lambda = \lambda(\sigma,\gamma) >0$ such that 
\begin{equation}\label{set_0}
\S_{\lambda}(T) \ls \cs \se_{2}^\sigma(0).
\end{equation}
\end{thm}
\begin{proof}

Throughout the proof we will employ constants of the form $\cs$ and $\ks$.  We recall that the meaning of these is stated in Section \ref{sec_notation}.

We first assume that $\delta_0$ is as small as in Proposition \ref{infty_ests}.  We also assume that $\delta_0$ is as small as in Proposition \ref{Hn_estimate} so that $\abs{\H_{2}} \ls (\se_{2}^0)^{3/2}$.   

We use Theorems \ref{energy_comparision}, \ref{dissipation_comparision}, and \ref{nlin_G} with $n = 2$ to bound 
\begin{equation}\label{set_1}
 \se_{2}^\sigma \ls \ks \left( (\seb_{2}^\sigma + \mathcal{H}_{2}) + \se_{2}^0 \se_{2}^\sigma + (\se_{2}^0)^{3/2} + \k \f_{2} \right) \text{ and } \sd_{2}^\sigma \ls \ks \left( \sdb_{2} + \se_{2}^0 \sd_{2}^\sigma + \k \f_{2} \right).
\end{equation}
Thus, if $\delta_0$ is small enough we may absorb the terms $\se_{2}^0 \se_{2}^\sigma + (\se_{2}^0)^{3/2}$ and $\se_{2}^0 \sd_{2}^\sigma$ terms onto the left to see that
\begin{equation}\label{set_2}
 \se_{2}^\sigma \ls \ks \left( \left( \seb_{2}^\sigma + \mathcal{H}_{2} \right) + \k \f_{2} \right) \text{ and } \sd_{2}^\sigma \ls \ks \left(\sdb_{2} + \k \f_{2} \right).
\end{equation}
Next we use Proposition \ref{kf_ests_st} to bound the product $\k \f_{2}$: 
\begin{equation}\label{set_{2}}
 \k \f_{2} \ls \cs \se_{2}^\sigma \min\{ \se_{2}^\sigma, \sd_{2}^\sigma\}.
\end{equation}
Plugging this into \eqref{set_2} then shows that 
\begin{equation}
 \se_{2}^\sigma \ls \cs \left( \left( \seb_{2}^\sigma + \mathcal{H}_{2} \right) + \se_{2}^\sigma \se_{2}^\sigma \right) \text{ and } \sd_{2}^\sigma \ls \ks \left(\sdb_{2} + \se_{2}^\sigma \sd_{2}^\sigma \right),
\end{equation}
and thus if we further restrict the size of $\delta_0$ we may again use an absorbing argument to conclude that 
\begin{equation}\label{set_4}
 \seb_{2}^\sigma + \mathcal{H}_{2} \le \se_{2}^\sigma \ls \cs (\seb_{2}^\sigma + \mathcal{H}_{2} ) \text{ and } \sdb_{2} \le \sd_{2}^\sigma \ls \cs \sdb_{2}.
\end{equation}

Next we use Theorem \ref{energy_dissipation_est} with $n = 2$ to see that 
\begin{equation}
\frac{d}{dt} \left( \seb_{2}^\sigma  + \mathcal{H}_{2} \right) + \sdb_{2}  \ls \gamma \sd_{2}^0 +
\ks \left[ \sqrt{\se_{2}^0 } \sd_{2}^\sigma  +   \sqrt{  \sd_{2}^\sigma \k \f_{2}} \right] \ls
\gamma \cs \sdb_{2} + \cs  \sqrt{\se_{2}^\sigma } \sdb_{2}^\sigma  
\end{equation}
where in the last inequality we have used \eqref{set_{2}} and \eqref{set_4}.  We may then further restrict the size of $\delta_0$ and $\gamma_0$ in order to absorb the terms on the right onto the left.  This yields the inequality 
\begin{equation}\label{set_5}
\frac{d}{dt} \left( \seb_{2}^\sigma  + \mathcal{H}_{2} \right) + \hal \sdb_{2}  \le 0. 
\end{equation}

By definition we have the estimate $\se_{2}^\sigma \le \cs \sd_{2}^\sigma$.  Using this with \eqref{set_4}, we may then bound  
\begin{equation}\label{set_6}
 \hal \sdb_{2} \ge \frac{1}{4 \cs} \sd_{2}^\sigma + \frac{1}{4 \cs} \se_{2}^\sigma \ge \frac{1}{4 \cs} \sd_{2}^\sigma + \lambda \left( \seb_{2}^\sigma  + \mathcal{H}_{2} \right)
\end{equation}
for a constant $\lambda = \lambda(\sigma,\gamma) >0$.  Combining \eqref{set_5} with the differential inequality \eqref{set_5} then shows that
\begin{equation}
 \frac{d}{dt} \left( \seb_{2}^\sigma  + \mathcal{H}_{2} \right)  + \lambda \left( \seb_{2}^\sigma  + \mathcal{H}_{2} \right) + \frac{1}{4\cs} \sd_{2}^\sigma  \le 0.
\end{equation}
This differential inequality may be integrated to see that 
\begin{equation}
e^{\lambda t} \left( \seb_{2}^\sigma(t)  + \mathcal{H}_{2}(t) \right) + \frac{1}{4\cs} \int_0^t e^{\lambda r} \sd_{2}^\sigma(r) dr \le \left( \seb_{2}^\sigma(0)  + \mathcal{H}_{2}(0) \right).
\end{equation}
Once again appealing to \eqref{set_4}, we deduce from this that 
\begin{equation}
 \sup_{0 \le t \le T} e^{\lambda t} \se_{2}^\sigma(t) + \int_0^T \sd_{2}^\sigma(t) dt \ls \cs \se_{2}^\sigma(0),
\end{equation}
which is \eqref{set_0}.

\end{proof}

\subsection{Proof of main result }

We now parlay Theorem \ref{apriori_st} into the proof of Theorem \ref{gwp_st}.  

\begin{proof}[Proof of Theorem \ref{gwp_st}]
The proof follows that of Theorem \ref{gwp_st}, combining the local existence result with the a priori estimates of Theorem \ref{apriori_st} and a continuation argument.  We again refer to \cite{guo_tice_per} or \cite{jang_tice_wang_gwp} for details.
\end{proof}

\appendix

\section{Elliptic estimates }\label{app_elliptics}

Here we record basic elliptic estimates for the capillary and Stokes operators.

\subsection{Capillary operator}

Consider the elliptic problem 
\begin{equation}\label{capillary_problem}
 -\sigma \Delta \psi + g \psi = f  \text{ on } \T^n.
\end{equation}
If $f \in H^{-1}(\T^n) = (H^1(\T^n))^\ast$, then a weak solution is readily found with a standard application of Riesz's representation theorem: there exists a unique $\psi \in H^1(\T^n)$ such that 
\begin{equation}
\int_{\T^n} g \psi \varphi + \sigma \nabla \psi \cdot \nabla \varphi = \br{f,\varphi} \text{ for all } \varphi \in H^1(\T^n).
\end{equation}

\begin{thm}\label{cap_elliptic}
Let $s \ge 0$ and suppose that $f \in H^s(\T^n) \hookrightarrow H^{-1}(\T^n)$.  Let $\psi \in H^1(\T^n)$ be the weak solution to \eqref{capillary_problem}.  Then $\psi \in H^{s+2}(\T^n)$ and 
\begin{equation}\label{cap_elliptic_01}
 \norm{\psi}_{s} \le \frac{1}{g} \norm{f}_s \text{ and } \norm{D^{2+s} \psi}_0 \ls \frac{1}{\sigma} \norm{D^{s} f}_{0},
\end{equation}
where $D = \sqrt{-\Delta}$.  Moreover, if $\int_{\T^n} \psi =0$, then 
\begin{equation}\label{cap_elliptic_02}
 \norm{\psi}_{s+2} \ls \frac{1}{\sigma} \norm{D^s f}_{0}.
\end{equation}
\end{thm}
\begin{proof}
We begin with the proof of \eqref{cap_elliptic_01}.  For $k \in \N$ and $r \in \R$ let $\Lambda_k^r$ denote the operator whose Fourier symbol is $(1+ \abs{\xi}^2)^{r/2} \vchi_{B(0,k)}(\xi)$.  We use $\Lambda_k^{2s} \psi$  as a test function and use the obvious properties of the operator $\Lambda_k^{2s}$ to arrive at the bound 
\begin{equation}
 g \ns{\Lambda_k^s \psi}_0 + \sigma \ns{\nab \Lambda_k^s \psi}_0 = \int_{\T^n} \Lambda_k^s f \Lambda_k^s \psi \le \frac{1}{2g} \ns{\Lambda_k^s f}_0 + \frac{g}{2} \ns{\Lambda_k^s \psi}_0.
\end{equation}
From this we deduce that 
\begin{equation}
 \ns{\Lambda_k^s \psi}_0 \le \frac{1}{g^2} \ns{\Lambda_k^s f}_0,
\end{equation}
and the first bound in \eqref{cap_elliptic_01} then follows by sending $k \to \infty$.  

To prove the second bound in \eqref{cap_elliptic_01} we employ the operator $\dot{\Lambda}_k^{r}$, whose Fourier symbol is $\abs{\xi}^r  \vchi_{B(0,k)}(\xi)$.  Using $\dot{\Lambda}_k^{2+2s} \psi$ and arguing as above, but absorbing with the $\sigma$ term on the left, then allows us to deduce the second bound in \eqref{cap_elliptic_01}.  The bound \eqref{cap_elliptic_02} follows by using  $\Lambda_k^{2+2s} \psi$ as a test function and applying the Poincar\'{e} inequality.
\end{proof}

\subsection{Perturbed Stokes with Dirichlet conditions}

Consider the problem 
\begin{equation}\label{stokes_dir_prob}
\begin{cases}
s \p_1 u  + \diverge S(p,u)  = f^1 & \text{in }\Omega \\
\diverge u =f^2 &\text{in }\Omega \\
u = f^3 & \text{on }  \Sigma \\
u =0 & \text{on } \Sigma_b.
\end{cases}
\end{equation}

The estimates for solutions are recorded in the following result, the proof of which is standard and thus omitted.

\begin{thm}\label{stokes_dirichlet_est}
Let $m \in \mathbb{N}$.  If $f^1 \in H^m(\Omega)$, $f^2 \in H^{m+1}(\Omega)$, and $f^3 \in H^{m+3/2}(\Sigma)$, then the solution pair $(u,p)$ to \eqref{stokes_dir_prob} satisfies  $u \in H^{m+2}(\Omega)$, $\nab p \in H^{m+1}(\Omega)$, and we have the estimate
\begin{equation}
 \norm{u}_{m+2} + \norm{\nab p}_{m} \ls \norm{f^1}_{m} +\norm{f^2}_{m+1}+ \norm{f^3}_{m+3/2}.
\end{equation}
\end{thm}

\subsection{Perturbed Stokes with stress conditions}

Consider the problem 
\begin{equation}\label{stokes_stress_prob}
\begin{cases}
s \p_1 u  + \diverge S(p,u)  = f^1 & \text{in }\Omega \\
\diverge u =f^2 &\text{in }\Omega \\
S(p,u) e_3  = f^3 & \text{on }  \Sigma \\
u =0 & \text{on } \Sigma_b.
\end{cases}
\end{equation}

The estimates for solutions are recorded in the following result, the proof of which is standard and thus omitted.

\begin{thm}\label{stokes_stress_est}
Let $m \in \mathbb{N}$.  If $f^1 \in H^m(\Omega)$, $f^2 \in H^{m+1}(\Omega)$, and $f^3 \in H^{m+1/2}(\Sigma)$,  then the solution pair $(u,p)$ to \eqref{stokes_stress_prob} satisfies $u \in H^{m+2}(\Omega)$, $p \in H^{m+1}(\Omega)$, and we have the estimate
\begin{equation}
 \norm{u}_{m+2} + \norm{p}_{m+1} \ls \norm{f^1}_{m} +\norm{f^2}_{m+1}+ \norm{f^3}_{m+1/2}.
\end{equation}
\end{thm}

\section{Analytic tools}

\subsection{Products in Sobolev spaces}
Here we record some product estimates in Sobolev spaces.

\begin{lem}\label{sobolev_product_1}
The following hold on $\Sigma$ and on $\Omega$.
\begin{enumerate}
 \item Let $0\le r \le s_1 \le s_2$ be such that  $s_1 > n/2$.  Let $f\in H^{s_1}$, $g\in H^{s_2}$.  Then $fg \in H^r$ and
\begin{equation}\label{s_p_01}
 \norm{fg}_{H^r} \lesssim \norm{f}_{H^{s_1}} \norm{g}_{H^{s_2}}.
\end{equation}

\item Let $0\le r \le s_1 \le s_2$ be such that  $s_2 >r+ n/2$.  Let $f\in H^{s_1}$, $g\in H^{s_2}$.  Then $fg \in H^r$ and
\begin{equation}\label{s_p_02}
 \norm{fg}_{H^r} \lesssim \norm{f}_{H^{s_1}} \norm{g}_{H^{s_2}}.
\end{equation}

\item Let $0\le r \le s_1 \le s_2$ be such that  $s_2 >r+ n/2$. Let $f \in H^{-r}(\Sigma),$ $g \in H^{s_2}(\Sigma)$.  Then $fg \in H^{-s_1}(\Sigma)$ and
\begin{equation}\label{s_p_03}
 \norm{fg}_{-s_1} \ls \norm{f}_{-r} \norm{g}_{s_2}.
\end{equation}
\end{enumerate}
\end{lem}
\begin{proof}
 See for example Lemma A.1 of \cite{guo_tice_per}.
\end{proof}

\subsection{Poisson integral}

Suppose that $\Sigma = (L_1 \mathbb{T}) \times (L_2 \mathbb{T})$.  We define the Poisson integral in $\Omega_- = \Sigma \times (-\infty,0)$ by
\begin{equation}\label{poisson_def_per}
\mathcal{P} f(x) = \sum_{n \in   (L_1^{-1} \mathbb{Z}) \times (L_2^{-1} \mathbb{Z}) }  e^{2\pi i n \cdot x'} e^{2\pi \abs{n}x_3} \hat{f}(n),
\end{equation}
where for $n \in   (L_1^{-1} \mathbb{Z}) \times (L_2^{-1} \mathbb{Z})$ we have written
\begin{equation}
 \hat{f}(n) = \int_\Sigma f(x')  \frac{e^{-2\pi i n \cdot x'}}{L_1 L_2} dx'.
\end{equation}
It is well known that $\mathcal{P}: H^{s}(\Sigma) \rightarrow H^{s+1/2}(\Omega_-)$ is a bounded linear operator for $s>0$.   We now show that how derivatives of $\mathcal{P} f$ can be estimated in the smaller domain $\Omega$.

\begin{lem}\label{p_poisson}
Let $\mathcal{P} f$ be the Poisson integral of a function $f$ that is either in $\dot{H}^{q}(\Sigma)$ or $\dot{H}^{q-1/2}(\Sigma)$ for $q \in \mathbb{N}$.  Then
\begin{equation} 
 \ns{\nab^q \mathcal{P}f }_{0} \ls \norm{f}_{\dot{H}^{q-1/2}(\Sigma)}^2 \text{ and }  \ns{\nab^q \mathcal{P}f }_{0} \ls \norm{f}_{\dot{H}^{q}(\Sigma)}^2.
\end{equation}
\end{lem}

\begin{proof}
See Lemma A.3 in \cite{guo_tice_per}
\end{proof}

We will also need $L^\infty$ estimates.
\begin{lem}\label{p_poisson_2}
Let $\mathcal{P} f$ be the Poisson integral of a function $f$ that is in $\dot{H}^{q+s}(\Sigma)$ for $q\ge 1$ an integer and $s> 1$.  Then
\begin{equation} 
 \ns{\nab^q \mathcal{P}f }_{L^\infty} \ls \ns{f}_{\dot{H}^{q+s}}.
\end{equation}
The same estimate holds for $q=0$ if $f$ satisfies $\int_{\Sigma} f =0$.
\end{lem}
\begin{proof}
See Lemma A.4 in \cite{guo_tice_per}
\end{proof}


\end{document}